\documentclass[a4paper,12pt,reqno]{amsart}
\usepackage[latin1]{inputenc}
\usepackage[english]{babel}
\usepackage{amsmath, amsthm, amssymb, amsopn, amsfonts, amstext, stmaryrd, enumerate, color, mathtools, hyperref, mathrsfs, enumitem, url, comment}
\usepackage[margin=1.1in]{geometry}
\usepackage[final]{showkeys}
\numberwithin{equation}{section}

\hypersetup{
	colorlinks=true,
	linkcolor=blue,
	citecolor=blue,
	urlcolor=black,
	linktoc=all
}

\newcommand{\N}{\mathbb{N}}

\newcommand{\R}{\mathbb{R}}

\newcommand{\loc}{{\rm loc}}

\newcommand{\Mat}{{\mbox{\normalfont Mat}}}

\newcommand{\dist}{{\mbox{\normalfont dist}}}

\newcommand{\defi}{{\mbox{\normalfont def}}}

\DeclareMathOperator{\diam}{diam}

\DeclareMathOperator{\osc}{osc}

\DeclareMathOperator{\Id}{Id}

\def\Xint#1{\mathchoice
{\XXint\displaystyle\textstyle{#1}}%
{\XXint\textstyle\scriptstyle{#1}}%
{\XXint\scriptstyle\scriptscriptstyle{#1}}%
{\XXint\scriptscriptstyle\scriptscriptstyle{#1}}%
\!\int}
\def\XXint#1#2#3{{\setbox0=\hbox{$#1{#2#3}{\int}$ }
\vcenter{\hbox{$#2#3$ }}\kern-.6\wd0}}

\def\dashint{\Xint-}

\theoremstyle{plain}
\newtheorem{definition}{Definition}[section]
\newtheorem{theorem}[definition]{Theorem}

\newtheorem{lemma}[definition]{Lemma}
\newtheorem{corollary}[definition]{Corollary}
\newtheorem{example}[definition]{Example}

\theoremstyle{definition}
\newtheorem{remark}[definition]{Remark}

\renewcommand{\le}{\leqslant}
\renewcommand{\leq}{\leqslant}
\renewcommand{\ge}{\geqslant}
\renewcommand{\geq}{\geqslant}

\begin{document}

\title[A quantitative version of the Gidas-Ni-Nirenberg Theorem]{A quantitative version of the\\ Gidas-Ni-Nirenberg Theorem}
\date{\today}

\author{Giulio Ciraolo}
\author{Matteo Cozzi}
\author{Matteo Perugini}
\author{Luigi Pollastro}

\address{
\vspace{-\baselineskip}
\newline
\textit{Giulio Ciraolo}
\newline
Universit\`a degli Studi di Milano, Dipartimento di Matematica, Via Saldini 50, 20133 Milan, Italy
\newline
\textit{E-mail address}: \textit{\tt giulio.ciraolo@unimi.it}
}

\address{
\vspace{-\baselineskip}
\newline
\textit{Matteo Cozzi}
\newline
Universit\`a degli Studi di Milano, Dipartimento di Matematica, Via Saldini 50, 20133 Milan, Italy
\newline
\textit{E-mail address}: \textit{\tt matteo.cozzi@unimi.it}
}

\address{
\vspace{-\baselineskip}
\newline
\textit{Matteo Perugini}
\newline
Universit\`a degli Studi di Milano, Dipartimento di Matematica, Via Saldini 50, 20133 Milan, Italy
\newline
\textit{E-mail address}: \textit{\tt matteo.perugini@unimi.it}
}

\address{
\vspace{-\baselineskip}
\newline
\textit{Luigi Pollastro}
\newline
Universit\`a degli Studi di Milano, Dipartimento di Matematica, Via Saldini 50, 20133 Milan, Italy \& Universit\`a degli Studi di Torino, Dipartimento di Matematica, Via Carlo Alberto 10, 10123 Turin, Italy
\newline
\textit{E-mail address}: \textit{\tt luigi.pollastro@unito.it}
}

\begin{abstract}
A celebrated result by Gidas, Ni \& Nirenberg asserts that classical positive solutions to semilinear equations~$- \Delta u = f(u)$ in a ball vanishing at the boundary must be radial and radially decreasing. In this paper we consider small perturbations of this equation and study the quantitative stability counterpart of this result.
\end{abstract}

\maketitle
	
\section{Introduction}

\noindent
In the celebrated paper~\cite{GNN79}, Gidas, Ni \& Nirenberg established the radial symmetry and monotonicity of solutions to semilinear equations set in a ball. More precisely, they obtained the following result. Here,~$B_1$ denotes the~$n$-dimensional unit ball centered at the origin and~$n \ge 2$ is an integer.

\begin{theorem}[\cite{GNN79}] \label{thmGNN}
Let~$f: [0, +\infty) \to \R$ be a locally Lipschitz continuous function and~$u \in C^2(B_1) \cap C^0(\overline B_1)$ be a solution of
\begin{equation} \label{GNNprob}
\begin{cases}
- \Delta u = f(u) & \quad \mbox{in } B_1, \\
u > 0 & \quad \mbox{in } B_1, \\
u = 0 & \quad \mbox{on } \partial B_1.
\end{cases}
\end{equation}
Then,~$u$ is radially symmetric and strictly decreasing in the radial direction.
\end{theorem}

This result had an immense impact on the PDE community. To establish it, the authors used the method of moving planes, which was first pioneered by Alexandrov~\cite{A58} to characterize embedded hypersurfaces with constant mean curvature (the so-called Alexandrov's soap bubble theorem) and later adapted by Serrin~\cite{S71} to obtain symmetry results for overdetermined problems.

After its publication, Theorem~\ref{thmGNN} has been extended in several directions. To name just a few of these generalizations, Berestycki \& Nirenberg~\cite{BN91} and Dancer~\cite{D92} weakened the notion of solution and provided a simplified proof, while Damascelli \& Pacella~\cite{DP98} and Damascelli \& Sciunzi~\cite{DS04} extended the result to the case of the~$p$-Laplace operator. On the other hand, different proofs have been devised in order to consider non-Lipschitz nonlinearities. In this regard, Brock~\cite{B98} dealt with a class of continuous nonlinearities by using the continuous Steiner symmetrization, while Dolbeault \& Felmer~\cite{DF04} and Dolbeault, Felmer \& Monneau~\cite{DFM05} developed to this aim a local version of the moving planes method. Lions~\cite{L81}, Kesavan \& Pacella~\cite{KP94}, and Serra~\cite{S13} even allowed for some discontinuous~$f$ through the aid of integral methods based on the isoperimetric inequality and Pohozaev's identity. See also the more recent~\cite{DPV22} by Dipierro, Poggesi \& Valdinoci, where integral methods have been employed to address nonlinear weighted equations in convex sectors, possibly set in an anisotropic framework.

In the present work we consider a perturbed version of problem~\eqref{GNNprob}, namely
\begin{equation} \label{mainprob}
\begin{cases}
- \Delta u = \kappa f(u) & \quad \mbox{in } B_1, \\
u > 0 & \quad \mbox{in } B_1, \\
u = 0 & \quad \mbox{on } \partial B_1,
\end{cases}
\end{equation}
for some continuously differentiable function~$\kappa: B_1 \to [0, +\infty]$. When~$f$ is non-negative and~$\kappa$ is radially symmetric and decreasing, solutions of~\eqref{mainprob} are also radially symmetric and decreasing. This was already observed by Gidas, Ni \& Nirenberg---see~\cite[Theorem~$1^\prime$]{GNN79}, which actually holds for a general right-hand side~$f = f(r, u)$ ($r = |x|$,~$x \in B_1$) decreasing in~$r$ but not necessarily non-negative. Our goal here is to establish a quantitative version of this result, obtaining an estimation of the asymmetry of~$u$ in terms of a quantity measuring the defect of~$\kappa$ from being radially symmetric and decreasing. To this aim, we introduce the \emph{deficit}
$$
\defi(\kappa) := \| \nabla^T \! \kappa \|_{L^\infty(B_1)} + \| \partial_r^+ \! \kappa \|_{L^\infty(B_1)},
$$
where~$\partial_r^+$ denotes the positive part of the radial derivative~$\partial_r := \frac{x}{|x|} \cdot \nabla$ (i.e.,~$\partial_r^+ \! \kappa := \max \left\{ 0, \partial_r \kappa \right\}$), while~$\nabla^T := \nabla - \frac{x}{|x|} \, \partial_r$ indicates the angular gradient. Observe that the deficit of~$\kappa$ vanishes if and only if~$\kappa$ is radially symmetric and non-decreasing. Our first result is the following.

\begin{theorem} \label{mainthm}
Let~$f: [0, +\infty) \to \R$ be a non-negative locally Lipschitz continuous function and~$\kappa \in C^1(\overline{B_1})$ be a non-negative function. Let~$u \in C^2(B_1) \cap C^0(\overline B_1)$ be a solution of~\eqref{mainprob} satisfying
\begin{equation} \label{Linftyuundercontrol}
\frac{1}{C_0} \le \| u \|_{L^\infty(B_1)} \le C_0,
\end{equation}
for some constant~$C_0 \ge 1$. Then,
\begin{equation} \label{ualmostradsymm}
|u(x) - u(y)| \le C \, \defi (\kappa)^\alpha \quad \mbox{for all } x, y \in B_1 \mbox{ such that } |x| = |y|
\end{equation}
and
\begin{equation} \label{uquasimono}
    \partial_r u(x) < C \, \defi(\kappa)^\alpha \quad \mbox{for all } x \in B_1 \setminus \{ 0 \},
\end{equation}
for some constants~$\alpha \in (0, 1]$ and~$C > 0$ depending only on~$n$,~$\| f \|_{C^{0, 1}([0, C_0])}$,~$\|\kappa\|_{L^\infty(B_1)}$, and~$C_0$.
\end{theorem}

Theorem~\ref{mainthm} is a quantitative version of the classical result by Gidas, Ni \& Nirenberg. We mention that, from another point of view, a different quantitative variant has been obtained by Rosset~\cite{R94}, who considered space-independent semilinear equations set in small perturbations of the unit ball. In Corollary~\ref{corol2} below we also give a result in this direction.

In Theorem~\ref{mainthm}, the proximity of the solution~$u$ to a radial configuration is measured through the deficit~$\defi(\kappa)$. Clearly, if~$\defi(\kappa)=0$ then~$u$ is radially symmetric and strictly decreasing, coherently with~\cite[Theorem~$1^\prime$]{GNN79}. Note that the strictness of the inequality in~\eqref{uquasimono} is meaningful only when~$\defi(\kappa) = 0$, whereas it obviously plays no role in the case~$\defi(\kappa) > 0$, the most relevant for the present paper.

We point out that, if one is only interested in the statement concerning the almost radial symmetry of~$u$, then weaker notions of deficit can be considered. Indeed, as confirmed by a careful analysis of the proof of Theorem~\ref{mainthm}, estimate~\eqref{ualmostradsymm} continues to hold if~$\defi(\kappa)$ is replaced by the zero-th order quantity
\begin{equation} \label{0orderdeficit}
\sup_{r \in (0, 1)} \underset{\partial B_r}{\osc} \, \kappa + \sup_{e \in \partial B_1} \sup_{0 \le \rho < r < 1} \big( \kappa(r e) - \kappa(\rho e) \big)_+.
\end{equation}
That is, the presence in~$\defi(\kappa)$ of a first order quantity such as the gradient of~$\kappa$ is only required to obtain the almost monotonicity statement~\eqref{uquasimono}. See Remark~\ref{0orderremark} at the end of Section~\ref{refsec} for more details.

We also point out that we required the nonlinearity~$f$ to be non-negative for the sole purpose of recovering exact symmetry when the right-hand side is radially symmetric and decreasing with respect to the space variable~$x$. It is readily checked that, if~$f$ changes sign, our proof can still be applied in its essence and gives the following result.
\begin{theorem} \label{refthm}
Let~$f: [0, +\infty) \to \R$ be locally Lipschitz continuous and~$\kappa \in C^1(\overline{B_1})$ be a non-negative function. Let~$u \in C^2(B_1) \cap C^0(\overline B_1)$ be a solution of~\eqref{mainprob} satisfying
\begin{equation} \label{uundercontrol}
\frac{1}{C_0} \big( {1 - |x|} \big) \le u(x) \le C_0 \quad \mbox{for all } x \in B_1.
\end{equation}
for some constant~$C_0 \ge 1$. Then,
$$
|u(x) - u(y)| \le C \, \| \nabla \kappa \|_{L^\infty(B_1)}^\alpha \quad \mbox{for all } x, y \in B_1 \mbox{ such that } |x| = |y|
$$
and
$$
    \partial_r u(x) < C \, \| \nabla \kappa \|_{L^\infty(B_1)}^\alpha \quad \mbox{for all } x \in B_1 \setminus \{ 0 \},
$$
for some constants~$\alpha \in (0, 1]$ and~$C > 0$ depending only on~$n$,~$\| f \|_{C^{0, 1}([0, C_0])}$,~$\|\kappa\|_{L^\infty(B_1)}$, and~$C_0$.
\end{theorem}
Notice that assumption~\eqref{uundercontrol} is a stronger version of~\eqref{Linftyuundercontrol} and that the almost symmetry of~$u$ is measured here in terms of the~$L^\infty$ norm of the gradient of~$\kappa$ instead of the more precise~$\defi(\kappa)$.

The original proof of Theorem~\ref{thmGNN} is done via the method of moving planes, which in turns heavily relies on the maximum principle. Our proof of Theorem~\ref{mainthm} is based on a quantitative version of this method. Hence, in it we replace the maximum principle with quantitative counterparts, such as the weak Harnack inequality and the ABP estimate.

A quantitative approach to the method of moving planes was started by Aftalion, Busca \& Reichel~\cite{ABR99}, who applied it to Serrin's overdetermined problem to obtain a~$\log$ type quantitative estimate of radial symmetry. Later, it was used by the first author and Vezzoni~\cite{CV18} to provide a sharp (linear) quantitative version of Alexandrov's soap bubble theorem. Other related results can be found in~\cite{CDPPV,CFMN18,CMS16}---see also~\cite{CR18} and references therein. 

The deficit of~$\kappa$ appears in estimates~\eqref{ualmostradsymm} and~\eqref{uquasimono} raised to some unspecified power~$\alpha$. The value we get for~$\alpha$ depends on a constant occurring in the weak Harnack inequality---it is essentially equal to~$\frac{1}{1 + \beta}$ with~$\beta$ as in Lemma~\ref{harlem}, see Remark~\ref{remark_alpha}. We believe it could be an interesting future line of investigation to determine the sharpest value of~$\alpha$, at least for some specific choice of the nonlinearity~$f$.

In order to obtain quantitative information on the asymmetry of the solution~$u$, we needed to understand also its boundedness and positivity in a quantitative way. We did this through the imposition of assumption~\eqref{Linftyuundercontrol}. As a consequence, the constants~$\alpha$ and~$C$ appearing in estimates~\eqref{ualmostradsymm}-\eqref{uquasimono} depended on the constant~$C_0$. Our next result provides conditions on the nonlinearity~$f$ which ensure the possibility of removing assumption~\eqref{Linftyuundercontrol} and thus making estimates~\eqref{ualmostradsymm}-\eqref{uquasimono} independent of the size of~$u$.

\begin{corollary} \label{maincor}
Let~$f: [0, +\infty) \to \R$ be a non-negative locally Lipschitz function satisfying the following two conditions:
\begin{enumerate}[label=$(\alph*)$,leftmargin=*]
\item \label{cor-a} Either~$f(0) > 0$ or~$f(s) \le A s^p$ for all~$s \in [0, s_0]$ and for some~$A, s_0 > 0$,~$p > 1$;
\item \label{cor-b} Either~$f(s) \le B s^{q_1}$ for all~$s \ge s_1$ and for some~$B, s_1 > 0$,~$q_1 \in (0, 1)$, or the limit~$\displaystyle\lim_{s \rightarrow +\infty} \frac{f(s)}{s^{q_2}}$ exists finite and positive for some~$q_2 \in \left( 1, 2^\star - 1 \right)$.\footnote{With the understanding that, if~$n = 2$, the exponent~$q_2$ can be any number strictly larger than~$1$.}
\end{enumerate}
Let~$\kappa \in C^1(\overline{B_1})$ be a function satisfying~$\kappa \ge \kappa_0$ in~$B_1$, for some constant~$\kappa_0 > 0$. Then, there exist two constants~$C > 0$ and~$\alpha \in (0, 1)$, depending only on~$n$,~$f$,~$\| \kappa \|_{L^\infty(B_1)}$, and~$\kappa_0$, such that any solution~$u \in C^2(B_1) \cap C^0(\overline{B_1})$ of~\eqref{mainprob} satisfies~\eqref{ualmostradsymm} and \eqref{uquasimono}.
\end{corollary}

Corollary~\ref{maincor} follows almost immediately from Theorem~\ref{mainthm}. To obtain it, we simply verify that, under conditions~\ref{cor-a} and~\ref{cor-b}, any solution of~\eqref{mainprob} satisfies the bounds~\eqref{Linftyuundercontrol} for some uniform constant~$C_0$.

Condition~\ref{cor-a} prescribes the growth of the nonlinearity~$f$ near the origin and is therefore the main tool needed to establish the lower bound on the~$L^\infty$ norm of~$u$ in~\eqref{uundercontrol}. We point out that~\ref{cor-a} allows virtually all possible behaviors for Lipschitz nonlinearities besides linear growth. On the other hand, condition~\ref{cor-b} concerns the behavior of~$f$ at infinity and is thus connected with the upper bound in~\eqref{Linftyuundercontrol}. Again, linear growth is excluded---both by the strict sublinearity assumption and by the Gidas-Spruck type asymptotic superlinearity/subcriticality prescription. We stress that this exception is to be expected, since for, say,~$f(s) = \lambda_1 s$, with~$\lambda_1$ being the first Dirichlet eigenvalue of~$B_1$, no bound like~\eqref{Linftyuundercontrol} can hold with uniform constant~$C_0$.

Theorem~\ref{mainthm} is a particular case of a broader result, in which a perturbation of the Laplace operator is considered alongside a more general space-dependent nonlinearity. We introduce this framework here below.

Let~$L$ be a second order elliptic operator of the type 
\begin{equation} \label{defiopellittico}
    L[ v ] := - \mathrm{Tr}\big(A \,  D^2 v \big) + b \cdot \nabla v, 
\end{equation}
where~$A \in C^{1, \theta}(\overline{B}_1; \Mat_n(\R))$ and~$b \in C^{1, \theta}(\overline{B}_1; \R^n)$, for some~$\theta \in (0, 1)$. We assume that~$A(x)$ is symmetric for every~$x \in B_1$ and that
\begin{equation} \label{hponAandb}
\begin{gathered}
\| A \|_{C^{1, \theta}(B_1)} + \| b \|_{C^{1, \theta}(B_1)} \le \Lambda, \\
\sum_{i, j = 1}^n A_{i j}(x) \xi_i \xi_j \ge \frac{1}{\Lambda} \, |\xi|^2 \quad \mbox{for all } x \in B_1 \mbox{ and } \xi \in \R^n,
\end{gathered}
\end{equation}
for some constant~$\Lambda \ge 1$. Our main goal is to provide a quantitative symmetry result for the problem
\begin{equation} \label{introperturbedprob}
\begin{cases}
    L[u]= g(\cdot,u) \qquad &{\rm{in}} \ B_1\\
    u > 0 \qquad & {\rm{in}} \ B_1,\\
    u = 0 \qquad & {\rm{on}} \ \partial B_1,
\end{cases}
\end{equation}
where~$g \in C^{1, \theta}( \overline{B}_1 \times [0,+\infty) )$ is a non-negative function. In order to do this, we introduce a new deficit which quantifies how much the differential operator~$L$ differs from the Laplacian and how far~$g$ is from a nonlinearity which assures radial symmetry of the solution. More precisely, given any real number~$U \ge 0$ we define
\begin{equation}\label{defdefiL}
\defi (L, g, U):= \| A - I_n \|_{C^{0, 1}(B_1)} + \| b \|_{C^{0, 1}(B_1)} + G(g, U),
\end{equation}
where~$I_n \in \Mat_n(\R)$ is the identity matrix and
$$
G(g, U) := \sup_{s \in [0, U]} \|\nabla_x^T g(\cdot, s) \|_{L^\infty(B_1)} + \sup_{s \in [0, U]} \|\partial_r^+ g(\cdot, s) \|_{L^\infty(B_1)}.
$$

\begin{theorem} \label{mainthm2}
Given~$\theta \in (0, 1)$, let~$A \in C^{1, \theta}(\overline{B}_1; \Mat_n(\R))$ and~$b \in C^{1, \theta}(\overline{B}_1; \R^n)$ be satisfying~\eqref{hponAandb}, for some~$\Lambda \ge 1$, and let~$g \in C_\loc^{1, \theta} \big( {\overline{B}_1 \times [0,+\infty)} \big)$ be a non-negative function. Let~$u \in C^2(B_1) \cap C^0(\overline{B_1})$ be a solution of~\eqref{introperturbedprob}, with~$L$ given by~\eqref{defiopellittico}. 

Given any constant~$C_0 \ge 1$, there exist two other constants~$\alpha \in (0, 1)$ and~$C > 0$, depending only on~$n$,~$\theta$,~$\Lambda$,~$C_0$, and on an upper bound on~$\| g \|_{C^{1, \theta}(\overline B_1 \times [0, C_0])}$, such that if~$u$ satisfies~\eqref{Linftyuundercontrol}, then
\begin{equation} \label{upealmostradsymm}
|u(x) - u(y)| \le C \, \defi (L, g, C_0)^\alpha \quad \mbox{for all } x, y \in B_1 \mbox{ such that } |x| = |y|
\end{equation}
and
\begin{equation} \label{ualmostraddecr}
\partial_r u(x) < C \, \defi(L, g, C_0)^\alpha \quad \mbox{for all } x \in B_1 \setminus \{ 0 \}.
\end{equation}
\end{theorem}

We mention that Theorem \ref{mainthm2} can be applied for instance when one considers semilinear problems in a small normal perturbation of the ball, as done in \cite{R94}. Indeed, we have the following corollary.

\begin{corollary} \label{corol2}
Given~$0 < \epsilon \le \epsilon_0$ and~$\theta \in (0, 1)$, let~$\Psi_\epsilon: \overline{B}_1 \to \R^n$ be an invertible map of class~$C^{3,\theta}$ such that~$\|\Psi_\epsilon - \Id \|_{C^{3,\theta}(B_1)} + \|\Psi_\epsilon^{-1} - \Id \|_{C^{3,\theta}(\Omega_\epsilon)} \leq \epsilon$, where~$\Omega_\epsilon := \Psi_\epsilon(B_1)$. Let~$u \in C^2 (\Omega_\epsilon) \cap C^0 (\overline\Omega_\epsilon)$ be a solution of
\begin{equation}
\label{problemperturbedball}
\begin{cases}
- \Delta u = f(u) & \quad \mbox{in } \Omega_\epsilon, \\
u>0 & \quad \mbox{in } \Omega_\epsilon, \\
u=0 & \quad \mbox{on } \partial \Omega_\epsilon,
\end{cases}
\end{equation}
where~$f \in C_\loc^{1, \theta} ([0, +\infty))$ is a nonnegative function, and assume that~$u$ satisfies
\begin{equation} \label{LinftyboundsonOmegaeps}
\frac{1}{C_0}  \le \|u\|_{L^\infty{(\Omega_\epsilon)}} \le C_0,
\end{equation}
for some constant~$C_0 \ge 1$. Then, there exist two other constants~$\alpha \in (0, 1)$ and~$C > 0$, depending only on~$n$,~$\theta$,~$f$,~$\epsilon_0$, and~$C_0$, such that
$$
|u(x) - u(y)| \le C \, \epsilon^\alpha \quad \mbox{for all } x, y \in \Omega_\epsilon \mbox{ such that } \big|{\Psi_\epsilon^{-1}(x)}\big| = \big|{\Psi_\epsilon^{-1}(y)}\big|.
$$
\end{corollary}

At the end of Section~\ref{sect_pert} we will exploit Corollary~\ref{corol2} to show approximate symmetry results for semilinear problems set in small perturbations of the unit ball, by means of a couple of examples.

\medskip 

The paper is organized as follows. In Section~\ref{sect_prelim} we recall a few known facts and prove some preliminary results. Section~\ref{sect_proofmain} is devoted to the proof of the main result of the paper, Theorem~\ref{mainthm}. The very brief Sections~\ref{refsec} and~\ref{maincorsec} are respectively occupied by the proofs of Theorem~\ref{refthm} and Corollary~\ref{maincor}. The paper is closed by Section~\ref{sect_pert}, which contains the proofs of Theorem~\ref{mainthm2} and Corollary~\ref{corol2}.

\subsection*{Acknowledgements}
The authors have been partially supported by the ``Gruppo Nazionale per l'Analisi Matematica, la Probabilit\`a e le loro Applicazioni'' (GNAMPA) of the ``Istituto Nazionale di Alta Matematica'' (INdAM, Italy). The first author has also been supported by the Research Project of the Italian Ministry of University and Research (MUR) Prin 2022 ``Partial differential equations and related geometric-functional inequalities'', grant number 20229M52AS\_004. The second author has also been supported by the Spanish grant PID2021-123903NB-I00 funded by MCIN/AEI/10.13039/501100011033 and by ERDF ``A way of making Europe''.

\section{Preliminary results} \label{sect_prelim}

\noindent
In this section we collect some known results that will be used later. We begin by recalling the following version of the ABP estimate, due to Cabr\'e~\cite{C95}.

\begin{lemma}[\cite{C95}] \label{ABPlem}
Let~$\Omega \subset \R^n$ be a bounded domain,~$c: \Omega \to \R$ be a measurable function such that~$c \ge 0$ a.e.~in~$\Omega$, and~$h \in L^n(\Omega)$. If~$v \in C^0(\overline{\Omega}) \cap C^2(\Omega)$ satisfies
$$
- \Delta v + c v \le h \quad \mbox{in } \Omega,
$$
then
$$
\sup_{\Omega} v \le \sup_{\partial \Omega} v_+ + C |\Omega|^{\frac{1}{n}} \| h_+ \|_{L^n(\Omega)},
$$
for some dimensional constant~$C \ge 1$.
\end{lemma}

Next, we have the following weak Harnack inequality. Given~$\delta > 0$ and~$\Omega \subset \R^n$, we write~$\Omega_\delta$ to indicate the set of points at distance more than~$\delta$ from the boundary of~$\Omega$, that is
$$
\Omega_\delta := \Big\{ {x \in \Omega : \dist(x, \partial \Omega) > \delta} \Big\}.
$$

\begin{lemma} \label{harlem}
Let~$\Omega \subset \R^n$ be a bounded convex domain. Denote by~$d$ the inradius of~$\Omega$---i.e., the radius of the largest ball contained in~$\Omega$---and assume that
\begin{equation} \label{in-diambound}
d \ge c_\sharp \diam(\Omega),
\end{equation}
for some constant~$c_\sharp \in (0, 1]$. Let~$\delta \in \left( 0, \frac{d}{3} \right]$,~$c \in L^\infty(\Omega)$, and~$h \in L^n(\Omega)$. If~$v \in C^2(\Omega)$ is a non-negative function satisfying
$$
- \Delta v + c v \ge h \quad \mbox{in } \Omega,
$$
then, it holds
\begin{equation} \label{harnack-thesis}
\sup_{p \in \Omega_\delta} \bigg( {\, \dashint_{B_{\frac{\delta}{2}}(p)} v^s \, dx} \bigg)^{\! \frac{1}{s}} \le C \left( \frac{d}{\delta} \right)^{\! \beta}  \left( \inf_{\Omega_\delta} v + \| h_- \|_{L^n(\Omega)} \right),
\end{equation}
for some positive constants~$s$,~$\beta$, and~$C$ depending only on~$n$,~$c_\sharp$, and on upper bounds on~$\| c_+ \|_{L^\infty(\Omega)}$ and~$\diam(\Omega)$.
\end{lemma}

The key point here is the power-like dependence on~$\delta$ of the constant appearing in~\eqref{harnack-thesis}. This occurs thanks to the \emph{controlled} convexity of~$\Omega$, in the sense of~\eqref{in-diambound}---in a general bounded domain the optimal dependence might instead be exponential. This phenomenon is possibly known to the expert reader and has been observed in~\cite{CMV16} for the full Harnack inequality.

\begin{proof}[Proof of Lemma~\ref{harlem}]
First of all, we observe that~$v$ satisfies~$- \Delta v + c_+ v \ge - h_-$ in~$\Omega$. Hence, the standard weak Harnack inequality---see, e.g.,~\cite[Theorem~9.22]{GT01}---, yields that, if~$B_r(q)$ is a ball such that~$B_{2 r}(q) \subset \Omega$, then
\begin{equation} \label{harnack_brick}
\bigg( {\, \dashint_{B_r(q)} v^s \, dx} \bigg)^{\! \frac{1}{s}} \le C_\star \left( \inf_{B_r(q)} v + r \| h_- \|_{L^n(B_{2 r}(q))} \right),
\end{equation}
for some constants~$s > 0$ and~$C_\star \ge 2$ depending only on~$n$ and on upper bounds on~$\| c_+ \|_{L^\infty(\Omega)}$ and~$\diam(\Omega)$.
	
Let~$p_0 \in \Omega$ be a point for which~$\dist(p_0, \partial \Omega) = d$. Since~$\delta < d$, we have that~$p_0 \in \Omega_{\delta}$. For any~$p \in \Omega_{\delta}$, we claim that
\begin{align}
\label{claimharnacka}
\bigg( {\, \dashint_{B_\frac{d}{2}(p_0)} v^s \, dx} \bigg)^{\! \frac{1}{s}} & \le C_\flat \left( \frac{d}{\delta} \right)^{\! \frac{\beta}{2}} \bigg( {\inf_{B_{\frac{\delta}{2}}(p)} v + \| h_- \|_{L^n(\Omega)}} \bigg), \\
\label{claimharnackb}
\bigg( {\, \dashint_{B_{\frac{\delta}{2}}(p)} v^s \, dx} \bigg)^{\! \frac{1}{s}} & \le C_\flat \left( \frac{d}{\delta} \right)^{\! \frac{\beta}{2}} \bigg( {\inf_{B_{\frac{d}{2}}(p_0)} v + \| h_- \|_{L^n(\Omega)}} \bigg),
\end{align}
for some constants~$C_\flat \ge 1$ and~$\beta > 0$ depending only on~$n$,~$c_\sharp$, and on upper bounds on~$\| c_+ \|_{L^\infty(\Omega)}$ and~$\diam(\Omega)$. It is clear that these estimates immediately lead to~\eqref{harnack-thesis}.

We only prove the validity of~\eqref{claimharnacka}, since~\eqref{claimharnackb} can be established in a completely analogous fashion. To establish~\eqref{claimharnacka}, we suppose, after a rigid movement, that~$p_0 = 0$ and~$p = \ell e_n$, where~$\ell := |p - p_0|$. We initially assume that~$2 \ell \ge d$. Denote by~$\mathscr{C}$ and~$\mathscr{C}'$ respectively the convex hulls of~$B_d(p_0) \cup B_\delta(p)$ and~$B_{d/2}(p_0) \cup B_{\delta/2}(p)$, i.e.,
$$
\mathscr{C} = \bigcup_{t \in [0, 1]} B_{(1 - t) d + t \delta}(t \ell e_n) \quad \mbox{and} \quad \mathscr{C}' = \bigcup_{t \in [0, 1]} B_{\frac{(1 - t) d + t \delta}{2}}(t \ell e_n).
$$
By the convexity of~$\Omega$, we have that~$\mathscr{C}' \subset \mathscr{C} \subset \Omega$. Consider now the recursive sequence
$$
\begin{dcases}
t_k := \dfrac{d}{2 \ell} + \dfrac{2 \ell - d + \delta}{2 \ell } \,  t_{k - 1} & \quad \mbox{for } k \in \N, \\
t_0 = 0,
\end{dcases}
$$
as well as the balls
$$
B_k := B_{r_k}(p_k) \quad \mbox{and} \quad B_k' := B_{\frac{r_k}{2}}(p_k),
$$
where~$r_k := (1 - t_k) d + t_k \delta$ and~$p_k := t_k \ell e_n$, for all~$k \in \N \cup \{ 0 \}$. It is easy to see that~$\{ t_k \}$ is a sequence of non-negative numbers, strictly increasing to~$\frac{d}{d - \delta}$ as~$k \rightarrow +\infty$.
Clearly,~$B_k \subset \mathscr{C}$ and~$B_k' \subset \mathscr{C}'$ for every~$k \in \N \cup \{ 0 \}$ such that~$t_k \in [0, 1]$. Furthermore, one checks that~$B_{\frac{r_k}{4}} \! \left( p_k - \frac{t_k}{4} e_n \right) \subset B_{k - 1}' \cap B_k'$ for every~$k \in \N$ and thus
\begin{equation} \label{intersecisfat}
\frac{|B_k'|}{|B_{k - 1}' \cap B_k'|} \le 2^n \quad \mbox{for all } k \in \N.
\end{equation}
Let~$N \in \N$ be the largest integer for which~$t_N \in [0, 1)$. We claim that
\begin{equation} \label{Nupperbound}
N \le C \log \left( \frac{d}{\delta} \right),
\end{equation}
for some constant~$C > 0$ depending only on~$c_\sharp$. To obtain~\eqref{Nupperbound}, we observe that~$t_k$ is explicitly given by
$$
t_k = \frac{d}{2 \ell} \sum_{j = 0}^{k - 1} \left( \frac{2 \ell - d + \delta}{2 \ell} \right)^j = \frac{d}{d - \delta} \left[ 1 - \left( \frac{2 \ell - d + \delta}{2 \ell} \right)^{\! k} \right],
$$
for all~$k \in \N \cup \{ 0 \}$. Hence, the condition~$t_N < 1$ is equivalent to the inequality
$$
N < \frac{\log \left( \frac{d}{\delta} \right)}{\log \left( 1 + \frac{d - \delta}{2 \ell - d + \delta} \right)}.
$$
By~\eqref{in-diambound} and the fact that~$\delta \le \frac{d}{3}$, we see that~$\frac{d - \delta}{2 \ell - d + \delta} \ge \frac{c_\sharp}{3}$, from which~\eqref{Nupperbound} follows.

We now use the fundamental weak Harnack inequality~\eqref{harnack_brick} to compare the~$L^s$ norms of~$v$ over two consecutive balls in the chain~$B_{k - 1}'$ and~$B_k'$---recall that~$B_k \subset \Omega$. By also taking into account~\eqref{intersecisfat}, we compute
\begin{align*}
\bigg( {\, \dashint_{B_{k - 1}'} v^s \, dx} \bigg)^{\! \frac{1}{s}} & \le C_\star \left( \inf_{B_{k - 1}'} v + r_{k - 1} \| h_- \|_{L^n(B_{k - 1})} \right) \le C_\star \left( \inf_{B_{k - 1}' \cap B_k'} v + d \, \| h_- \|_{L^n(\Omega)} \right) \\
& \le C_\star \left\{ \bigg( {\, \dashint_{B_{k - 1}' \cap B_k'} v^s \, dx} \bigg)^{\! \frac{1}{s}} + d \, \| h_- \|_{L^n(\Omega)} \right\} \\
& \le C_\star \left\{ 2^n \bigg( {\, \dashint_{B_k'} v^s \, dx} \bigg)^{\! \frac{1}{s}} + d \, \| h_- \|_{L^n(\Omega)} \right\},
\end{align*}
for every~$k \in \{1, \ldots, N\}$. By chaining these estimates, we find that
\begin{equation} \label{weakhartech}
\bigg( {\, \dashint_{B_{\frac{d}{2}}(p_0)} v^s \, dx} \bigg)^{\! \frac{1}{s}} \le \left( 2^n C_\star \right)^N \max \{2 d , 1\} \left\{ \bigg( {\dashint_{B_N'} v^s \, dx} \bigg)^{\! \frac{1}{s}} + \| h_- \|_{L^n(\Omega)} \right\}.
\end{equation}
Arguing as before---using now that~$\big| {B_N' \cap B_{\delta/2}(p)} \big| \ge \big| {B_{\delta/4} \! \left( p - \frac{\delta}{4} e_n \right)} \big| \ge 2^{-n} \big| {B_{\delta/2}(p)} \big|$ and estimate~\eqref{harnack_brick} a couple of times---we obtain
\begin{equation} \label{weakhartech2}
\bigg( {\, \dashint_{B_N'} v^s \, dx} \bigg)^{\! \frac{1}{s}} \le 2^n C_\star^2  \, \max \{2 d, 1\} \bigg( {\inf_{B_{\frac{\delta}{2}}(p)} v + \| h_- \|_{L^n(\Omega)}} \bigg).
\end{equation}
By combining this with~\eqref{weakhartech} and recalling the upper bound~\eqref{Nupperbound} on~$N$, estimate~\eqref{claimharnacka} readily follows, under the assumption that~$2 \ell \ge d$.

When~$2 \ell < d$, the computation is less involved, as the balls~$B_{d/2}(p_0)$ and~$B_{\delta/2}(p)$ have large intersection---of measure comparable to~$\big| {B_{\delta/2}(p)} \big|$. In view of this,~\eqref{claimharnacka} can be deduced at once by arguing exactly as for~\eqref{weakhartech2}. The proof is thus complete.
\end{proof}

As an immediate consequence of the above lemma, we deduce the following Harnack inequality in the spherical dome
\begin{equation} \label{Sigmalambdadef}
\Sigma_\lambda := B_1 \cap \{ x_n > \lambda \},
\end{equation}
for~$\lambda \in [0, 1)$. Given~$\delta > 0$, we also consider the set
\begin{equation} \label{Omegadef}
\Sigma_{\lambda, \delta} := \Big\{ {x \in \Sigma_\lambda : \dist(x, \partial \Sigma_\lambda) > \delta} \Big\}.
\end{equation}

\begin{corollary} \label{harcor}
Let~$0 \le \lambda \le \lambda_0 < 1$,~$\delta \in \left( 0, \frac{1 - \lambda_0}{6} \right]$,~$c \in L^\infty(\Sigma_\lambda)$, and~$h \in L^n(\Sigma_\lambda)$. Let~$v \in C^2(\Sigma_\lambda)$ be a non-negative function satisfying
$$
- \Delta v + c v \ge h \quad \mbox{in } \Sigma_\lambda.
$$
Then,
$$
\sup_{p \in \Sigma_{\lambda, \delta}} \bigg( {\, \dashint_{B_{\frac{\delta}{2}}(p)} v^s \, dx} \bigg)^{\! \frac{1}{s}} \le \frac{C}{\delta^\beta} \left( \inf_{\Sigma_{\lambda, \delta}} v + \| h_- \|_{L^n(\Sigma_\lambda)} \right),
$$
for some positive constants~$s$,~$\beta$, and~$C$ depending only on~$n$,~$\lambda_0$, and on an upper bound on~$\| c_+ \|_{L^\infty(\Sigma_\lambda)}$.
\end{corollary}

\section{Proof of Theorem~\ref{mainthm}} \label{sect_proofmain}

\noindent
Our proof of Theorem~\ref{mainthm} is via the method of moving planes. Before getting into the argument, we make a few preliminary observations.

\subsection*{Step 1: Preliminary remarks}

First of all, as~$0 \le u \le C_0$,~$\kappa$ is bounded, and~$f$ is locally bounded, by standard elliptic estimates there exists a constant~$C_1 > 0$, depending only on~$n$,~$C_0$,~$\|\kappa\|_{L^\infty(B_1)}$, and~$\| f \|_{L^\infty([0, C_0])}$, for which
\begin{equation} \label{gradientbound}
\|\nabla u \|_{L^\infty(B_1)} + [ \nabla u ]_{C^{\frac{9}{10}}(B_1)} \le C_1.
\end{equation}
As it is customary,~$[v]_{C^{\tau}\!(E)}$ stands for the H\"older seminorm of order~$\tau \in (0, 1)$ of a function~$v = (v_1, \ldots, v_N): E \to \R^N$ defined on a set~$E \subset \R^n$ and with~$N \in \N$, that is
$$
[v]_{C^{\tau}\!(E)}:= \sup_{i = 1, \ldots, N} \, \sup_{\substack{x,y \in E\\ x \neq y}} \frac{|v_i(x) - v_i(y)|}{|x-y|^\tau}.
$$

Next, we claim that
\begin{equation} \label{uunivlingrow}
u(x) \ge \frac{1}{C_2} \big( {1 - |x|} \big) \quad \mbox{for all } x \in B_1,
\end{equation}
for some constant~$C_2 \ge 1$ depending only on~$n$,~$C_0$,~$\| \kappa \|_{L^\infty(B_1)}$, and~$\| f \|_{L^\infty([0, C_0])}$. To see this, we first remark that, taking advantage of estimates~\eqref{Linftyuundercontrol} and~\eqref{gradientbound}, one easily obtains the existence of a point~$p\in B_1$ and of a constant~$r \in \left( 0, \frac{1}{8} \right]$ such that
$$
u \geq \frac{1}{2C_0} \mbox{ in } B_{2 r}(p) \quad \mbox{and} \quad \dist \big( {B_{2 r}(p), \partial B_1} \big) \ge 2 r.
$$
Note that the constant~$r$ only depends on~$n$,~$C_0$,~$\| \kappa \|_{L^\infty(B_1)}$, and~$\| f \|_{L^\infty([0, C_0])}$. We now seek a lower bound for~$u$ over~$\overline{B_r}$. Clearly, we have that
$$
B_{2 r}(p)\subset B_{1-2 r}\subset B_{1 - r}(x)\subset B_1\quad  \mbox{for all } x \in \overline{B_r}.
$$
As~$f$ and~$\kappa$ are non-negative, the function~$u$ is superharmonic in~$B_1$. Hence, by the mean value theorem and the non-negativity of~$u$ in~$B_1$ we obtain that
\begin{equation} \label{ugecsharp}
u(x) \ge \dashint_{B_{1 - r}(x)} u(y) \, dy \ge \frac{1}{|B_{1 -r}|} \int_{B_{2 r}(p)} u(y) \, dy \ge \frac{2^{n - 1} r^n}{C_0} =: c_\sharp \quad \mbox{for all } x \in \overline{B_r}.
\end{equation}
From this and the weak maximum principle we get that~$u$ is larger than the unique continuous function which is harmonic in~$B_1 \setminus \overline{B_r}$, vanishes on~$\partial B_1$, and is equal to~$c_\sharp$ on~$\partial B_r$. Since this function is explicit---it is an appropriate affine transformation of the fundamental solution for the Laplacian in~$\R^n$---, we easily deduce that
$$
u(x) \ge c_\flat (1-|x|) \quad \mbox{for all } x \in B_1 \setminus \overline{B_r},
$$
for some constant~$c_\flat > 0$ depending only on~$n$,~$C_0$,~$\| \kappa \|_{L^\infty(B_1)}$, and~$\| f \|_{L^\infty([0, C_0])}$. Claim~\eqref{uunivlingrow} then readily follows from this estimate and~\eqref{ugecsharp}.

Finally, we suppose without loss of generality that
\begin{equation} \label{defikappalegamma}
\defi(\kappa) \in (0, \gamma],
\end{equation}
for some small~$\gamma \in (0, 1)$ to be chosen in dependence of~$n$,~$\| f \|_{C^{0, 1}([0, C_0])}$,~$\|\kappa\|_{L^\infty(B_1)}$, and~$C_0$ only. Indeed, the case~$\defi(\kappa) = 0$ corresponds to the classical Gidas-Ni-Nirenberg setting in which Theorem~\ref{mainthm} boils down to~\cite[Theorem~$1^\prime$]{GNN79}. If, on the other hand,~$\defi(\kappa) > \gamma$, then
$$
|u(x) - u(y)| \le |u(x)| + |u(y)| \le 2 C_0 \le \frac{2 C_0}{\gamma^\alpha} \, \defi(\kappa)^\alpha \quad \mbox{for all } x, y \in B_1
$$
and
$$
\partial_r u(x) \le |\nabla u(x)| \le C_1 \le \frac{C_1}{\gamma^\alpha} \, \defi(\kappa)^\alpha \quad \mbox{for all } x \in B_1 \setminus \{ 0 \},
$$
for any~$\alpha \in (0, 1]$. Hence, claims~\eqref{ualmostradsymm} and~\eqref{uquasimono} are trivially verified in this case.

Assuming~\eqref{defikappalegamma} to hold true, we proceed with the proof. Clearly, claim~\eqref{ualmostradsymm} is equivalent to showing that, given any unit vector~$e \in \partial B_1$, it holds
\begin{equation} \label{ualmostsymmindire}
u(x) - u(x^{(e)}) \le C \, \defi(\kappa)^\alpha \quad \mbox{for all } x \in B_1 \mbox{ such that } x \cdot e > 0,
\end{equation}
for some constants~$C \ge 1$ and~$\alpha \in (0, 1)$ depending only on~$n$,~$\| f \|_{C^{0, 1}([0, C_0])}$,~$\|\kappa\|_{L^\infty(B_1)}$, and~$C_0$, and where we write~$x^{(e)} := x - 2 (x \cdot e) e$ to indicate the symmetric point of~$x$ with respect to the hyperplane orthogonal to~$e$ passing through the origin. Up to a rotation, we may assume that~$e = e_n$---note that the rotation of~$u$ solves an equation for a possibly different~$\kappa$ which, however, still satisfies~\eqref{defikappalegamma}. Under this assumption,~\eqref{ualmostsymmindire} becomes
\begin{equation} \label{ualmostsymmindiren}
u(x', x_n) - u(x', - x_n) \le C \, \defi(\kappa)^\alpha \quad \mbox{for all } x \in B_1 \mbox{ such that } x_n > 0.
\end{equation}
We shall establish~\eqref{ualmostsymmindiren} in the next four steps. In a further step we will then tackle the almost radial monotonicity statement~\eqref{uquasimono}.

\subsection*{Step 2: Starting the moving planes procedure.}

Let~$\lambda \in (0, 1)$. Recalling definition~\eqref{Sigmalambdadef}, we consider the function
$$
w_\lambda(x) := u(x^\lambda) - u(x) \quad \mbox{for } x \in \Sigma_\lambda,
$$
where, for~$x = (x', x_n)$, we define
$$
x^\lambda := (x', 2 \lambda - x_n).
$$
Notice that~$w_\lambda$ is a solution of
\begin{equation} \label{eqforwlambda}
- \Delta w_\lambda + c_\lambda w_\lambda = f_\lambda \quad \mbox{in } \Sigma_\lambda,
\end{equation}
where
\begin{equation} \label{deficilambda}
c_\lambda(x) := \begin{dcases}
- \kappa(x) \, \frac{f(u(x^\lambda)) - f(u(x))}{u(x^\lambda) - u(x)} & \quad \mbox{if } u(x^\lambda) \ne u(x), \\
0 & \quad \mbox{if } u(x^\lambda) = u(x),
\end{dcases}
\end{equation}
and
\begin{equation} \label{defieffelambda}
f_\lambda(x) := \big( {\kappa(x^\lambda) - \kappa(x)} \big) f(u(x^\lambda)),
\end{equation}
for all~$x \in \Sigma_\lambda$.

In view of equation~\eqref{eqforwlambda}, we see that~$v := -w_\lambda$ satisfies
$$
\begin{cases}
- \Delta v + (c_\lambda)_+ v = - f_\lambda + (c_\lambda)_- v & \quad \mbox{in } \Sigma_\lambda, \\
v \le 0 & \quad \mbox{on } \partial \Sigma_\lambda.
\end{cases}
$$
Hence, Lemma~\ref{ABPlem} gives that
\begin{equation} \label{ABPtech1}
\sup_{\Sigma_\lambda} v \le C_3 |\Sigma_\lambda|^{\frac{1}{n}} \Big( {\| (- f_\lambda)_+ \|_{L^n(\Sigma_\lambda)} + \| (c_\lambda)_- v_+ \|_{L^n(\Sigma_\lambda)}} \Big),
\end{equation}
for some dimensional constant~$C_3 \ge 1$. Clearly,
\begin{equation} \label{clambdaest}
\begin{aligned}
\| (c_\lambda)_- v_+ \|_{L^n(\Sigma_\lambda)} & \le |\Sigma_\lambda|^{\frac{1}{n}} \| c_\lambda \|_{L^\infty(\Sigma_\lambda)} \| v_+ \|_{L^\infty(\Sigma_\lambda)} \\
& \le |\Sigma_\lambda|^{\frac{1}{n}} [ f ]_{C^{0, 1}([0, C_0])} \, \|\kappa\|_{L^\infty(B_1)} \sup_{\Sigma_\lambda} v_+.
\end{aligned}
\end{equation}
On the other hand, considering the auxiliary point~$\tilde{x}^\lambda := \frac{|x^\lambda|}{|x|} \, x$ and observing that
\begin{equation} \label{xtildexcollinear}
	\mbox{$x$ and~$\tilde{x}^\lambda$ belong to the same ray coming out of the origin}
\end{equation}
and
\begin{equation} \label{xtildexlambdax}
	|\tilde{x}^\lambda| = |x^\lambda| < |x|,
\end{equation}
for every~$x \in \Sigma_\lambda$ we estimate
\begin{equation} \label{RHSesttech}
\begin{aligned}
- f_\lambda(x) & = \Big( {\big( {\kappa(x) - \kappa(\tilde{x}^\lambda)} \big) + \big( {\kappa(\tilde{x}^\lambda) - \kappa(x^\lambda)} \big)} \Big) f(u(x^\lambda)) \\
& \le \| f \|_{L^\infty([0, C_0])} \bigg( {\sup_{e \in \partial B_1} \sup_{0 \le \rho < r < 1}  \big( {\kappa(r e) - \kappa(\rho e)} \big)_+ + \sup_{r \in (0, 1)} \underset{\partial B_r}{\osc} \, \kappa} \bigg) \\
& \le \pi \| f \|_{L^\infty([0, C_0])} \, \defi(\kappa).
\end{aligned}
\end{equation}
Note that we exploited the non-negativity of~$f$. Consequently,
\begin{equation} \label{-flambda+est}
\| (- f_\lambda)_+ \|_{L^n(\Sigma_\lambda)} \le \pi |\Sigma_\lambda|^{\frac{1}{n}} \| f \|_{L^\infty([0, C_0])} \, \defi(\kappa)
\end{equation}
and thus, recalling~\eqref{ABPtech1} and~\eqref{clambdaest},
$$
\sup_{\Sigma_\lambda} v \le \frac{C_4}{2} |\Sigma_\lambda|^{\frac{2}{n}} \Big( {\defi(\kappa) + \sup_{\Sigma_\lambda} v_+} \Big),
$$
with~$C_4 := 8 C_3 \left( 1 + \|\kappa\|_{L^\infty(B_1)} \right) \| f \|_{C^{0, 1}([0, C_0])}$. Now, if~$\sup_{\Sigma_\lambda} v = \sup_{\Sigma_\lambda} v_+ > 0$, then rearranging terms in the previous inequality we deduce that
\begin{equation} \label{ABPtech2}
\sup_{\Sigma_\lambda} v \le \defi(\kappa),
\end{equation}
provided~$C_4 |\Sigma_\lambda|^{\frac{2}{n}} \le 1$---which holds true, for instance, if~$\lambda \ge 1 - \left( C_4^{\frac{n}{2}} |B_1'| \right)^{-1}$. Since~\eqref{ABPtech2} is also trivially satisfied when~$\sup_{\Sigma_\lambda} v \le 0$, recalling the definition of~$v$ we conclude that
$$
w_\lambda \ge - \defi(\kappa) \quad \mbox{in } \Sigma_\lambda.
$$

Consider the set
\begin{equation} \label{Lambdadef}
\Lambda := \Big\{ {\lambda \in (0, 1) : w_\mu \ge - \defi(\kappa) \mbox{ in } \Sigma_\mu \mbox{ for all } \mu \in [\lambda, 1)} \Big\}.
\end{equation}
We just proved that~$\Lambda$ contains the interval~$[\lambda_0, 1)$, where
$$
\lambda_0 := \max \left\{ 1 - \left( C_4^{\frac{n}{2}} |B_1'| \right)^{-1}, \frac{1}{2} \right\}.
$$
Its infimum
\begin{equation} \label{lambdastardef}
\lambda_\star := \inf \Lambda
\end{equation}
is thus a well-defined real number lying in the interval~$[0, \lambda_0]$.

\subsection*{Step 3: Reaching an intermediate position.}

We first claim that
$$
\lambda_\star \le \frac{1}{4}.
$$
We argue by contradiction, and assume instead that~$\lambda_\star \in \left( \frac{1}{4}, \lambda_0 \right]$. It is immediate to see that~$\lambda_\star \in \Lambda$ and therefore that
$$
w_{\lambda_\star} \ge - \defi(\kappa) \quad \mbox{in } \Sigma_{\lambda_\star}.
$$
Let now~$v := w_{\lambda_\star} + \defi(\kappa)$. Clearly,~$v$ satisfies
\begin{equation} \label{eqfor-w}
\begin{cases}
- \Delta v + c_{\lambda_\star} v = f_{\lambda_\star} + c_{\lambda_\star} \defi(\kappa) & \quad \mbox{in } \Sigma_{\lambda_\star}, \\
v \ge 0 & \quad \mbox{in } \Sigma_{\lambda_\star}.
\end{cases}
\end{equation}
Therefore, we can apply to it Corollary~\ref{harcor} and, by taking advantage of estimate~\eqref{-flambda+est} with~$\lambda = \lambda_\star$, deduce that
\begin{equation} \label{estfor-w}
\sup_{p \in \Sigma_{\lambda_\star, \delta}} \bigg( {\, \dashint_{B_{\frac{\delta}{2}}(p)} v^s \, dx} \bigg)^{\! \frac{1}{s}} \le \frac{C_5}{\delta^\beta} \left( \inf_{\Sigma_{\lambda_\star, \delta}} v + \defi(\kappa) \right),
\end{equation}
for every~$\delta \in \left( 0, \frac{1 - \lambda_0}{6} \right]$, for three constants~$s > 0$,~$\beta > 0$, and~$C_5 \ge 1$ depending only on~$n$,~$\| f \|_{C^{0, 1}([0, C_0])}$, $\|\kappa\|_{L^\infty(B_1)}$, and with~$\Sigma_{\lambda_\star, \delta}$ as in~\eqref{Omegadef}. Recalling the linear growth estimate~\eqref{uunivlingrow}, the gradient bound in~\eqref{gradientbound}, and the fact that~$u$ vanishes on the boundary of~$B_1$, we now observe that
\begin{equation} \label{vtostech1}
\begin{aligned}
& \sup_{p \in \Sigma_{\lambda_\star, \delta}} \bigg( {\, \dashint_{B_{\frac{\delta}{2}}(p)} v^s \, dx} \bigg)^{\! \frac{1}{s}} \ge \bigg( {\, \dashint_{B_{\frac{\delta}{2}}((1 - 2 \delta) e_n)} v^s \, dx} \bigg)^{\! \frac{1}{s}} \ge \inf_{B_{\frac{\delta}{2}}((1 - 2 \delta) e_n)} v \\
& \hspace{30pt} = \inf_{x \in B_{\frac{\delta}{2}}((1 - 2 \delta) e_n)} \big( {u(x^{\lambda_\star}) - u(x)} \big) + \defi(\kappa) \\
& \hspace{30pt} \ge \frac{1}{C_2} \min \left\{ 1 - \left| 2 \lambda_\star - 1 + \frac{5}{2} \, \delta \right|, 1 - \left| 2 \lambda_\star - 1 + \frac{3}{2} \, \delta \right| \right\} - 3 \| \nabla u \|_{L^\infty(B_1)} \delta \\
& \hspace{30pt} \ge \frac{1 - \lambda_0}{2 C_2}  - 3 C_1 \delta
\end{aligned}
\end{equation}
As a result,
\begin{equation} \label{vtostech2}
\inf_{\Sigma_{\lambda_\star, \delta}} w_{\lambda_\star} \ge \frac{\delta^\beta}{C_5} \left( \frac{1 - \lambda_0}{2 C_2}  - 3 C_1 \delta \right) - 2 \, \defi(\kappa).
\end{equation}
Take now~$\delta := \min \left\{ \frac{1 - \lambda_0}{12 C_1 C_2}, \frac{(4 C_4)^{- \frac{n}{2}}}{n + 2} \right\} \in \left( 0, \frac{1 - \lambda_0}{6} \right]$ and assume that
$$
\gamma \le \frac{(1 - \lambda_0) \delta^\beta}{16 C_2 C_5}.
$$
By virtue of these choices, we easily deduce that~$w_{\lambda_\star} \ge \frac{(1 - \lambda_0) \delta^\beta}{8 C_2 C_5}$ in~$\Sigma_{\lambda_\star, \delta}$. Consequently,
$$
w_{\lambda_\star - \varepsilon}(x) \ge w_{\lambda_\star}(x) - 2 \| \nabla u \|_{L^\infty(B_1)} \varepsilon \ge \frac{(1 - \lambda_0) \delta^\beta}{8 C_2 C_5} - 2 C_1 \varepsilon \ge 0 \quad \mbox{for all } x \in \Sigma_{\lambda_\star, \delta},
$$
provided~$\varepsilon > 0$ is sufficiently small. Hence,~$v := - w_{\lambda_\star - \varepsilon}$ is a solution of
$$
\begin{cases}
- \Delta v + (c_{\lambda_\star - \varepsilon})_+ v = - f_{\lambda_\star - \varepsilon} + (c_{\lambda_\star - \varepsilon})_- v & \quad \mbox{in } \Sigma', \\
v \le 0 & \quad \mbox{on } \partial \Sigma',
\end{cases}
$$
where~$\Sigma' := \Sigma_{\lambda_\star - \varepsilon} \setminus \Sigma_{\lambda_\star, \delta}$. Taking advantage of Lemma~\ref{ABPlem}, we then get that
\begin{equation} \label{vtoreabsorb}
\sup_{\Sigma'} v \le \frac{C_4}{2} |\Sigma'|^{\frac{2}{n}} \left( \defi(\kappa) + \sup_{\Sigma'} v_+ \right).
\end{equation}
Observe that~$\Sigma' \subset \big( {\Sigma_{\lambda_\star - \varepsilon} \setminus \Sigma_{\lambda_\star + \delta}} \big) \cup \big( {\left( B_1 \setminus B_{1 - \delta} \right) \cap \left\{ x_n \ge \lambda_\star + \delta \right\}} \big)$ and thus that, recalling the definition of~$\delta$,
\begin{equation} \label{Sigma'est}
C_4 |\Sigma'|^{\frac{2}{n}} \le C_4 \big( {|B_1'| (\delta + \varepsilon) + |B_1 \setminus B_{1 - \delta}|} \big)^{\frac{2}{n}} \le 4 C_4 \big( {(n + 2) \delta} \big)^{\frac{2}{n}} \le 1,
\end{equation}
provided~$\varepsilon \le \delta$. As a result, we easily infer from inequality~\eqref{vtoreabsorb} that
$$
w_{\lambda_\star - \varepsilon} \ge - \defi(\kappa) \quad \mbox{in } \Sigma_{\lambda_\star - \varepsilon} \quad \mbox{for all } \varepsilon \in [0, \varepsilon_0],
$$
for some small~$\varepsilon_0 > 0$, contradicting the fact that~$\lambda_\star$ is the infimum of~$\Lambda$.

\subsection*{Step 4: Going almost all the way.}

We claim that
\begin{equation} \label{lambdastaralmost0}
\lambda_\star \le \left( 3 C_2 C_5 C_6^\beta \, \defi(\kappa) \right)^{\! \frac{1}{1 + \beta}},
\end{equation}
where~$C_6 := \max \left\{ 3 C_1 C_2, (4 C_4)^{\frac{n}{2}} (n + 2) \right\} \ge 2$. Notice that
$$
\left( 3 C_2 C_5 C_6^\beta \, \defi(\kappa) \right)^{\! \frac{1}{1 + \beta}} < \frac{1}{4},
$$
provided we take
$$
\gamma \le \frac{1}{4^{2 + \beta} C_2 C_5 C_6^\beta}.
$$
To establish~\eqref{lambdastaralmost0}, we argue once again by contradiction and suppose that
\begin{equation} \label{lastcontr}
\lambda_\star \in \left( \left( 3 C_2 C_5 C_6^\beta \, \defi(\kappa) \right)^{\! \frac{1}{1 + \beta}}, \frac{1}{4} \right].
\end{equation}
As before, we have that the function~$v := w_{\lambda_\star} + \defi(\kappa)$ satisfies~\eqref{eqfor-w} and therefore estimate~\eqref{estfor-w} for every~$\delta \in \left( 0, \frac{1}{8} \right]$, by Corollary~\ref{harcor}. Computing as for~\eqref{vtostech1}-\eqref{vtostech2}, we find that
$$
\inf_{\Sigma_{\lambda_\star, \delta}} w_{\lambda_\star} \ge \frac{\delta^\beta}{C_5} \left( \frac{2 \lambda_\star}{C_2} - 3 C_1 \delta \right) - 2 \, \defi(\kappa).
$$
Taking~$\delta := \min \left\{ \frac{\lambda_\star}{C_6}, \frac{(4 C_4)^{- \frac{n}{2}}}{n + 2} \right\} \in \left( 0, \frac{1}{8} \right]$, by recalling~\eqref{lastcontr} and the definition of~$C_6$ we get~$w_{\lambda_\star} \ge \frac{\lambda_\star^{1 + \beta}}{3 C_2 C_5 C_6^\beta}$ in~$\Sigma_{\lambda_\star, \delta}$. As a result,~$w_{\lambda_\star - \varepsilon} \ge 0$ in~$\Sigma_{\lambda_\star, \delta}$ if~$\varepsilon > 0$ is small enough and, arguing as before,
\begin{equation} \label{suptobereabsorbed}
\sup_{\Sigma'} \, (- w_{\lambda_\star - \varepsilon}) \le \frac{C_4}{2} \left| \Sigma' \right|^{\frac{2}{n}} \left( \defi(\kappa) + \sup_{\Sigma'} \, (- w_{\lambda_\star - \varepsilon})_+ \right),
\end{equation}
where~$\Sigma' := \Sigma_{\lambda_\star - \varepsilon} \setminus \Sigma_{\lambda_*, \delta}$. Since we still have the bound~\eqref{Sigma'est} on the measure of~$\Sigma'$---thanks to the definitions of~$\delta$ and~$C_6$, and provided we take~$\varepsilon \le \delta$---, it easily follows from inequality~\eqref{suptobereabsorbed} that
$$
w_{\lambda_\star - \varepsilon} \ge - \defi (\kappa) \quad \mbox{in } \Sigma_{\lambda_\star - \varepsilon} \quad \mbox{for all } \varepsilon \in [0, \varepsilon_0],
$$
for some small~$\varepsilon_0 > 0$. This contradicts the definition of~$\lambda_\star$ and thus~\eqref{lambdastaralmost0} holds true.

\subsection*{Step 5: Almost radial symmetry in one direction.}

Thus far, we have proved that
$$
u(x', x_n) - u(x', 2 \lambda - x_n) \le \defi (\kappa) \quad \mbox{for all } (x', x_n) \in \Sigma_\lambda \mbox{ and } \lambda \in [\lambda_1, 1),
$$
with
\begin{equation} \label{lambda1def}
\lambda_1 := \left( 3 C_2 C_5 C_6^\beta \, \defi(\kappa) \right)^{\! \frac{1}{1 + \beta}}.
\end{equation}
By choosing~$\lambda = \lambda_1$ and recalling the gradient bound~\eqref{gradientbound}, we get that
$$
u(x', x_n) - u(x', - x_n) \le C_7 \, \defi(\kappa)^{\frac{1}{1 + \beta}} \quad \mbox{for all } (x', x_n) \in \Sigma_{\lambda_1}.
$$
for some constant~$C_7 \ge 1$ depending only on~$n$,~$C_0$,~$\|\kappa\|_{L^\infty(B_1)}$, and~$\| f \|_{C^{0, 1}([0, C_0])}$. On top of this, we also have that, for~$(x', x_n) \in \Sigma_0 \setminus \Sigma_{\lambda_1}$,
\begin{align*}
u(x', x_n) - u(x', - x_n) & \le |u(x', x_n) - u(x', 0)| + |u(x', - x_n) - u(x', 0)| \\
& \le 2 \| \nabla u \|_{L^\infty(B_1)} x_n \le 2 C_1 \lambda_1 \le C_7 \, \defi(\kappa)^{\frac{1}{1 + \beta}},
\end{align*}
up to possibly taking a larger~$C_7$. The last two inequalities yield the validity of~\eqref{ualmostsymmindiren}.

\subsection*{Step 6: Almost monotonicity in the radial direction.}

Our goal is to show that
\begin{equation} \label{quasi_monotonia}
    \partial_n u(x) \le C \, \defi(\kappa)^{\frac{1}{1+\beta}} \quad \mbox{for all } x \in B_1 \mbox{ such that } x_n > 0,
\end{equation}
for some constant~$C > 0$ depending only on~$n$,~$C_0$,~$\|\kappa\|_{L^\infty(B_1)}$, and~$\| f \|_{C^{0, 1}([0, C_0])}$. It is clear that, by specializing this to the points~$x = (0, x_n)$ with~$x_n \in (0, 1)$ and up to a rotation, this yields~\eqref{uquasimono}.

Let~$\lambda \in [\lambda_1, 1)$ with~$\lambda_1$ defined as in~\eqref{lambda1def} and let~$\varepsilon > 0$ to be soon chosen small. Setting~$v:= -w_\lambda$ it holds 
\begin{equation*}
    L_\lambda v:= -\Delta v + c_\lambda v = - f_\lambda \quad \mbox{in } N_{\lambda, \varepsilon},
\end{equation*}
where~$c_\lambda$ and~$f_\lambda$ are as in~\eqref{deficilambda} and~\eqref{defieffelambda}, while~$N_{\lambda, \varepsilon} := \Sigma_\lambda \setminus \Sigma_{\lambda + \varepsilon}$. Note that, by definition~\eqref{Sigmalambdadef}, if
\begin{equation} \label{roundborderhyp}
\varepsilon \ge 1 - \lambda,
\end{equation}
then~$\Sigma_{\lambda + \varepsilon}$ is empty, in which case we simply have~$N_{\lambda, \varepsilon} = \Sigma_\lambda$. We plane to achieve~\eqref{quasi_monotonia} by constructing a supersolution for the operator~$L_\lambda$ in~$N_{\lambda, \varepsilon}$. In order to do this, we need appropriate estimates on the coefficient~$c_\lambda$ and the right-hand side~$- f_\lambda$.

Definition~\eqref{deficilambda} immediately yields
\begin{equation} \label{contostimaclambda}
    \|c_\lambda\|_{L^\infty(\Sigma_\lambda)} \le [f]_{C^{0,1}([0, C_0])} \, \|\kappa\|_{L^\infty(B_1)} =:\Gamma.
\end{equation}
From this it follows that there exists a constant~$\varepsilon_0 > 0$, depending only on~$[f]_{C^{0,1}([0, C_0])}$ and~$\|\kappa\|_{L^\infty(B_1)}$, such that
\begin{equation} \label{WMPok}
\mbox{the weak maximum principle holds for~$L_\lambda$ in~$N_{\lambda, \varepsilon}$},
\end{equation}
for every~$\varepsilon \in (0, \varepsilon_0]$, thanks to the maximum principle for narrow domains---see, e.g.,~\cite{BNV94} or~\cite[Section~3.3]{GT01}.

The estimate of~$- f_\lambda$ from above requires a bit more work. Given~$x \in \Sigma_\lambda$, let~$\tilde{x}^\lambda := \frac{|x^\lambda|}{|x|} \, x$ and recall that~\eqref{xtildexcollinear} and~\eqref{xtildexlambdax} hold true. We also have that
\begin{equation} \label{xtildexcontrol}
|x| - |\tilde{x}^\lambda| \le 2 (x_n - \lambda) \quad \mbox{and} \quad
\dist_{\, \partial B_{|x^\lambda|}} \! \left( \tilde{x}^\lambda, x^\lambda \right) \le 2 \pi (x_n - \lambda),
\end{equation}
where~$\dist_{\, \partial B_r}$ denotes the geodesic distance on the sphere of radius~$r > 0$---we can disregard the case~$|x^\lambda| = 0$ since, if this occurs, then~$\tilde{x}^\lambda = x^\lambda = 0$. The first inequality in~\eqref{xtildexcontrol} follows right away from the definition of~$\tilde{x}^\lambda$, while the second can be obtained noticing that~$\dist_{\, \partial B_r}(p, q) \le \frac{\pi}{2} |p - q|$ for all~$p, q \in \partial B_r$ and computing as follows:
\begin{align*}
|\tilde{x}^\lambda - x^\lambda|^2 & = 2 \left(|x^\lambda|^2 - \tilde{x}^\lambda \cdot x^\lambda \right) = 2 \, \frac{|x^\lambda|}{|x|} \left( |x| |x^\lambda| - x \cdot x^\lambda \right) \le 2 \, \frac{|x|^2 |x^\lambda|^2 - (x \cdot x^\lambda)^2}{|x| |x^\lambda| + x \cdot x^\lambda} \\
& = 2 \, \frac{\big( {|x'|^2 + x_n^2} \big) \big( {|x'|^2 + (2 \lambda - x_n)^2} \big) - \big( {|x'|^2 + x_n (2 \lambda - x_n)} \big)^2}{\sqrt{|x'|^2 + x_n^2} \sqrt{|x'|^2 + (2 \lambda - x_n)^2} + |x'|^2 + x_n (2 \lambda - x_n)} \\
& \le 8 \, \frac{|x'|^2 \left( x_n - \lambda \right)^2}{x_n |2 \lambda - x_n| + |x'|^2 + x_n (2 \lambda - x_n)} \le 8 (x_n - \lambda)^2.
\end{align*}
By virtue of~\eqref{xtildexcollinear},~\eqref{xtildexlambdax}, and~\eqref{xtildexcontrol}, recalling definition~\eqref{defieffelambda} we obtain
\begin{equation} \label{contostimaeffelambda}
\begin{aligned}
-f_\lambda(x) & = \Big( {\left( \kappa(x) - \kappa(\tilde{x}^\lambda) \right) + \left( \kappa(\tilde{x}^\lambda) - \kappa(x^\lambda) \right)} \Big) f(u(x^\lambda)) \\
& \le 2 \pi \| f \|_{L^\infty([0, C_0])} \big( {\| \partial_r^+ \! \kappa \|_{L^\infty(B_1)} + \| \nabla^T \! \kappa \|_{L^\infty(B_1)}} \big) (x_n - \lambda) \\
& = 2 \pi \| f \|_{L^\infty([0, C_0])} \, \defi(\kappa) (x_n - \lambda).
\end{aligned}
\end{equation}
Note that here we also took advantage of the fact that~$f$ is non-negative.

Now that we understood the sizes of~$c_\lambda$ and~$- f_\lambda$, we are in position to construct an upper barrier~$\bar{v}$ for the function~$v$ in~$N_{\lambda, \varepsilon}$. For~$M, \mu > 0$, let
\begin{equation*}
    \bar{v}(x) := M \sin \big( {\mu \, (x_n - \lambda)} \big).
\end{equation*}
Recalling~\eqref{contostimaclambda}, for every~$x \in N_{\lambda, \varepsilon}$ we have
$$
    L_\lambda \bar{v} (x) = (\mu^2 + c_\lambda) \, \bar{v} (x) \ge (\mu^2 - \Gamma) \, \bar{v} (x) \ge \bar{v} (x) \ge \frac{2 M \mu}{\pi} (x_n - \lambda),
$$
if we take~$\mu^2 \ge \Gamma + 1$ and~$\mu \, \varepsilon \le \pi/2$. Going back to~\eqref{contostimaeffelambda}, this gives that
\begin{equation} \label{barvsuper}
L_\lambda \bar{v} \ge - f_\lambda \quad \mbox{in } N_{\lambda, \varepsilon}, 
\end{equation}
provided~$M \mu \ge \pi^2 \| f \|_{L^\infty([0, C_0])} \, \defi(\kappa)$.

To deal with the boundary condition, we decompose~$\partial N_{\lambda, \varepsilon}$ as~$\partial N_{\lambda, \varepsilon} = D \cup T_\lambda \cup T_{\lambda + \varepsilon}$, where~$D := \{ \lambda \le x_n \le \lambda + \varepsilon \} \cap \partial B_1$ is the round part, while~$T_\lambda := \{ x_n = \lambda \} \cap B_1$ and~$T_{\lambda + \varepsilon} := \{ x_n = \lambda + \varepsilon \} \cap B_1$ are the flat parts---note that, if~\eqref{roundborderhyp} is satisfied, then~$T_{\lambda + \varepsilon} = \varnothing$ and~$\partial N_{\lambda, \varepsilon}$ is only made up of the round part~$D$ and a single flat part~$T_\lambda$. Observe that
$$
\begin{cases}
    v = 0 = \bar{v} & \quad \mbox{on} \ T_{\lambda},\\
    v < 0 < \bar{v} & \quad \mbox{on} \ D.
\end{cases}
$$
When~\eqref{roundborderhyp} does not hold, recalling definitions~\eqref{Lambdadef} and~\eqref{lambdastardef}, as well as the fact that~$\lambda \ge \lambda_\star$, thanks to~\eqref{lambdastaralmost0} and the way we took~$\lambda$, we also have
$$
v - \bar{v} \le \mathrm{def} (\kappa) -  \frac{ 2M \mu \, \varepsilon}{\pi} \le 0 \quad \mathrm{on} \ T_{\lambda + \varepsilon},
$$
provided~$M \mu \, \varepsilon \ge \frac{\pi}{2} \, \defi(\kappa)$. Thus, whether \eqref{roundborderhyp} is satisfied of not, we get that~$v \le \bar{v}$ on~$\partial N_{\lambda, \varepsilon}$.

Thanks to this,~\eqref{WMPok}, and~\eqref{barvsuper}, by setting
\begin{equation*}
	\mu:= \sqrt{\Gamma + 1}, \quad \varepsilon := \min \bigg\{ \frac{\varepsilon_0}{2}, \, \frac{\pi}{2 \mu} \bigg\}, \quad M := \frac{\pi}{\mu} \max \left\{ \pi \| f \|_{L^\infty([0, C_0])}, \, \frac{1}{2 \varepsilon} \right\} \defi(\kappa)
\end{equation*}
and applying the weak maximum principle we find that
\begin{equation*}
u(x) - u(x^\lambda) = v(x) \le \bar{v}(x) \le 2 \, C_8 \, \defi(\kappa) \, (x_n - \lambda) \quad \mbox{for all } x \in N_{\lambda, \varepsilon},
\end{equation*}
for some constant~$C_8 > 0$ depending only on~$n$,~$\| f \|_{C^{0,1}([0, C_0])}$, and~$\| \kappa \|_{L^\infty(B_1)}$. Letting~$x_n \to \lambda^+$ in the above relation, we get that~$\partial_n u(x) \le C_8 \, \defi(\kappa)$ for all~$x \in N_{\lambda, \varepsilon}$ and~$\lambda \in [\lambda_1, 1)$, that is
\begin{equation} \label{quasimonsigmalambda}
    \partial_n u (x) \le C_8 \, \defi(\kappa) \quad \mbox{for all } x \in \Sigma_{\lambda_1}.
\end{equation}
If instead~$x \in \Sigma_0 \setminus \Sigma_{\lambda_1}$, taking advantage of~\eqref{gradientbound} we get
\begin{equation} \label{quasimonsigmazero}
\begin{aligned}
    \partial_n u(x',x_n) & \le \partial_n u (x',\lambda_1) + (\lambda_1 - x_n)^{\frac{9}{10}} \, [ \nabla u ]_{C^{\frac{9}{10}}(B_1)} \\
    & \le C_8 \, \defi(\kappa) + C_1 \lambda_1^{\frac{9}{10}} \le C_{8} \, \mathrm{def} (\kappa)^{\frac{9}{10} \frac{1}{1+\beta}}, 
\end{aligned}
\end{equation}
where~$C_8 > 0$ only depends on~$n$,~$\| f \|_{C^{0, 1}([0, C_0])}$,~$\| \kappa \|_{L^\infty(B_1)}$, and~$C_0$. By putting together~\eqref{quasimonsigmalambda} and~\eqref{quasimonsigmazero}, we are led to~\eqref{quasi_monotonia}. The proof of Theorem~\ref{mainthm} is thus complete.

\begin{remark} \label{remark_alpha}
As already pointed out in the introduction, the value of the exponent~$\alpha$ appearing in estimates~\eqref{ualmostradsymm} and~\eqref{uquasimono} can be computed rather explicitly.

Concerning~\eqref{ualmostradsymm}, it is straightforward to verify that
\begin{equation} \label{alpha_value}
\alpha = \frac{1}{1 + \beta},
\end{equation}
where~$\beta$ is as in Lemma~\ref{harlem}.

The value of~$\alpha$ appearing in~\eqref{uquasimono} is slightly worse and can be made arbitrarily close to~$1/(1+\beta)$ at the cost of having a larger constant~$C$ in~\eqref{uquasimono}. Indeed,~\eqref{uquasimono} is obtained by exploiting the continuity of the gradient of~$u$. For simplicity, we took advantage here of the~$C^{1,\frac{9}{10}}$ regularity of~$u$, which implies 
$$ 
\alpha = \frac{9}{10} \frac{1}{1+\beta}
$$ 
in~\eqref{uquasimono}. However, it is clear that we can get~$\alpha$ arbitrarily close to~$1/(1+\beta)$ by using instead the~$C^{1,1-\epsilon}$ regularity of~$u$, with~$\epsilon > 0$ small enough. This naturally comes at the price of having a constant~$C$ in~\eqref{uquasimono} which blows up as~$\epsilon \rightarrow 0$.

It is also worth mentioning that~$\alpha$ can be actually chosen as~\eqref{alpha_value} also in~\eqref{uquasimono}, by employing~$C^2$ estimates for~$u$ via the Schauder theory. In this case, the constant~$C$ would also depend on (an upper bound for) a H\"older norm of~$\kappa$.
\end{remark}

\section{Sketch of the proof of Theorem~\ref{refthm}} \label{refsec}
\noindent
We include here a few brief indications on how to modify the strategy described in Section~\ref{sect_proofmain} in order to prove Theorem~\ref{refthm}.

First of all, the argument presented in Step~1 to deduce~\eqref{uunivlingrow}---which relied on the superharmonicity of~$u$---is no longer needed, in light of assumption~\eqref{uundercontrol}. Secondly, the specific expression of~$\defi(\kappa)$ comes into play only to estimate the right-hand side of equation~\eqref{eqforwlambda}. This is done at essentially two points of the proof: items~\eqref{RHSesttech}-\eqref{-flambda+est} in Step~2 and~\eqref{contostimaeffelambda} in Step~6. Recalling the definition~\eqref{defieffelambda} of~$f_\lambda$, it is immediate to see that~\eqref{-flambda+est} can be replaced by the simpler
\begin{align*}
\| f_\lambda \|_{L^n(\Sigma_\lambda)} & \le |\Sigma_\lambda|^{\frac{1}{n}} \| f \|_{L^\infty([0, C_0])} \osc_{B_1} \kappa \\
& \le 2 |\Sigma_\lambda|^{\frac{1}{n}} \| f \|_{L^\infty([0, C_0])} \| \nabla \kappa \|_{L^\infty(B_1)},
\end{align*}
and~\eqref{contostimaeffelambda} by
$$
|f_\lambda(x)| \le 2 \| f \|_{L^\infty([0, C_0])} \| \nabla \kappa \|_{L^\infty(B_1)} (x_n - \lambda) \quad \mbox{for all } x \in \Sigma_\lambda.
$$
Once these modifications are made, the rest of the argument goes through verbatim with~$\| \nabla \kappa \|_{L^\infty(B_1)}$ in place of~$\defi(\kappa)$.

\begin{remark} \label{0orderremark}
As already pointed out in the introduction, it is possible to have a version of Theorem~\ref{mainthm} (and Theorem~\ref{refthm}) involving a weaker zero-th order deficit, where only the almost radial symmetry of the solution~$u$ is inferred. This is the case since the first order structure of~$\defi(\kappa)$ (or~$\| \nabla \kappa \|_{L^\infty(B_1)}$) is used just to establish the almost radial monotonicity of~$u$.

Indeed, in the proof of Theorem~\ref{mainthm}, Steps~1-5 can be carried out with the zero-th order quantity appearing in~\eqref{0orderdeficit} in place of~$\defi(\kappa)$ and they would still lead to the almost radiality of~$u$---essentially, it suffices to disregard the last line in estimate~\eqref{RHSesttech}. A similar observation can be made for Theorem~\ref{refthm}, where the almost radial symmetry would still hold true if~$\| \nabla \kappa \|_{L^\infty(B_1)}$ is replaced by the weaker~$\osc_{B_1} \kappa$.
\end{remark}

\section{Proof of Corollary~\ref{maincor}} \label{maincorsec}

\noindent
In order to establish Corollary~\ref{maincor}, it suffices to show that, under conditions~\ref{cor-a} and~\ref{cor-b}, any solution~$u \in C^2(B_1) \cap C^0(\overline{B_1})$ of~\eqref{mainprob} satisfies the bounds~\eqref{Linftyuundercontrol} for some constant~$C_0 \ge 1$ depending only on~$n$,~$f$,~$\| \kappa \|_{L^\infty(B_1)}$, and~$\kappa_0$. Indeed, once this is obtained, the result immediately follows from Theorem~\ref{mainthm}.

\subsection*{Step 1: Uniform bound from above}
We begin by showing that~\ref{cor-b} yields the upper bound in~\eqref{Linftyuundercontrol}. Note that this is an immediate consequence of the classical a priori estimate of Gidas \& Spruck~\cite{GS81} when~$f$ is asymptotically subcritical and superlinear at infinity.

If, on the other hand,~$f$ is strictly sublinear at infinity, then~$f(s) \le B' (1 + s^{q_1})$ for all~$s \ge 0$, for some constants~$B' > 0$ and~$q_1 \in (0, 1)$. The claim now easily follows by testing the equation against~$u$. Indeed, by doing this we get
$$
\int_{B_1} |\nabla u|^2 \, dx = \int_{B_1} \kappa f(u) u \, dx \le B' \| \kappa \|_{L^\infty(B_1)} \left( \int_{B_1} u \, dx + \int_{B_1} u^{q_1 + 1} \, dx \right).
$$
By combining this estimate with the H\"older's, Young's, and Poincar\'e's inequalities, we easily get that~$\| u \|_{H^1(B_1)}$ is bounded by some constant depending only on~$n$,~$q_1$,~$B'$, and~$\| \kappa \|_{L^\infty(B_1)}$. By virtue of this estimate and, again, the sublinearity of~$f$, a uniform~$L^\infty(B_1)$ bound on~$u$ is obtained for instance by bootstrapping Calder\'on-Zygmund estimates.

\subsection*{Step 2: Uniform bound from below}
We now address the validity of the lower bound in~\eqref{Linftyuundercontrol} under assumption~\ref{cor-a}.

When~$f(0) > 0$, the result follows by comparison with an appropriate multiple of the torsion function. Indeed, by the continuity of~$f$ and the non-negativity of both~$f$ and~$\kappa$, there exists~$\epsilon \in (0, 1]$ such that~$\kappa(x) f(u) \ge \epsilon \, \chi_{[0, \epsilon]}(u)$ for every~$x \in B_1$ and~$u \ge 0$. Consequently,~$u$ is a supersolution of~$-\Delta u = \epsilon \chi_{[0, \epsilon]}(u)$ in~$B_1$. On the other hand, the function~$\underline{u}(x) := \frac{\epsilon}{2 n} (1 - |x|^2)$ is a solution of~$- \Delta \underline{u} = \epsilon \, \chi_{[0, \epsilon]}(\underline{u})$ in~$B_1$. From the comparison principle---note that the nonlinearity~$\epsilon \chi_{[0, \epsilon]}$ is non-increasing in~$[0, +\infty)$---we deduce that~$u \ge \underline{u}$ in~$B_1$, which gives the lower bound in~\eqref{Linftyuundercontrol}.

Suppose now that the alternative assumption holds in~\ref{cor-a}, namely that~$f(s) \le A s^p$ for all~$s \in [0, s_0]$, for some constants~$A, s_0 > 0$ and~$p > 1$. Without loss of generality, we may assume that~$s_0 \in (0, 1)$ and~$p \in (1, 2^\star - 1)$ when~$n \ge 3$. In order to get the bound from below in~\eqref{Linftyuundercontrol}, we first observe that either~$\| u \|_{L^\infty(B_1)} > s_0$---in which case we are done---or~$0 \le u(x) \le s_0$ for every~$x \in B_1$. Assuming the latter, we test the equation for~$u$ against~$u$ itself, obtaining
$$
\int_{B_1} |\nabla u|^2 \, dx = \int_{B_1} \kappa f(u) u \, dx \le A \| \kappa \|_{L^\infty(B_1)} \int_{B_1} u^{p + 1} \, dx.
$$
By the Poincar\'e-Sobolev and H\"older inequalities, we deduce from this that
$$
\| u \|_{L^{p + 1}(B_1)}^2 \le C_\sharp A \| \kappa \|_{L^\infty(B_1)} \| u \|_{L^{p + 1}(B_1)}^{p + 1},
$$
for some constant~$C_\sharp > 0$ depending only on~$n$ and~$p$. Since~$p > 1$, we can reabsorb to the right the~$L^{p + 1}(B_1)$ norm appearing on the left-hand side. By doing this, we get that~$\| u \|_{L^{p + 1}(B_1)} \ge \left( C_\sharp A \| \kappa \|_{L^\infty(B_1)} \right)^{\! \frac{1}{1 - p}}$. Our claim immediately follows from this and the proof is thus complete.

\section{Proof of Theorem~\ref{mainthm2} and Corollary~\ref{corol2}} \label{sect_pert}

\noindent
This section is devoted to the proof of Theorem~\ref{mainthm2} and Corollary~\ref{corol2}. At the end of it, we will add a couple of examples of application of this last result.

In Theorem~\ref{mainthm2} we provide a quantitative symmetry result for problem~\eqref{introperturbedprob}, measured in terms of the deficit defined in~\eqref{defdefiL}. Its proof is analogous to that of Theorem~\ref{mainthm}. For this reason, in what follows we only highlight the main differences.

\begin{proof}[Sketch of the proof of Theorem~\ref{mainthm2}]
As we just mentioned, the proof follows the lines of that of Theorem~\ref{mainthm}. Here we only outline the modifications in what corresponds there to Steps~1,~2, and~6---with the understanding that the remaining parts follow by analogous arguments.

Let~$g^\star > 0$ be such that~$\| g \|_{C^{1, \theta}(\overline B_1 \times [0, C_0])} \le g^\star$. We begin by observing that, since~$u$ satisfies~\eqref{Linftyuundercontrol}, by standard elliptic estimates (see, e.g.,~\cite[Theorem~6.19]{GT01}) there exists a constant~$C_1 > 0$, depending only on~$n$,~$\theta$,~$\Lambda$,~$g^\star$, and~$C_0$, for which
\begin{equation} \label{generalgradientbound}
\|\nabla u\|_{L^\infty(B_1)} + \| D^2 u \|_{L^\infty(B_1)} + \| D^3 u \|_{L^\infty(B_1)} \le C_1. 
\end{equation}
Furthermore, we assume without loss of generality that
$$
\defi (L,g, C_0) \le \gamma
$$
for some small~$\gamma \in (0, 1/2]$ to be chosen in dependence of~$n$,~$\theta$,~$\Lambda$,~$g^\star$, and~$C_0$ only.

The main goal of the first step is to show that 
\begin{equation}\label{eq: u has linear growth, GENERAL OPERATOR}
u(x) \ge \frac{1}{C_2} (1 - |x|) \quad \mbox{for all } x \in B_1,
\end{equation}
for some constant~$C_2 \ge 1$ depending only on~$n$,~$\theta$,~$\Lambda$,~$C_0$, and~$g_*$. Similarly to what we did when proving~\eqref{uunivlingrow}, we start by noticing that, thanks to~\eqref{Linftyuundercontrol} and~\eqref{generalgradientbound}, there exists a point~$p_0\in B_1$ and a constant~$r\in (0,\frac{1}{4}]$ such that
\begin{align}\label{eq: u strettamente positiva in una pallina dentro B1}
    u\geq \frac{1}{2C_0}\textnormal{ in }B_{r}(p_0)\quad \textnormal{and}\quad \dist(B_{r}(p_0),\partial B_1)\geq r.
\end{align}
From this we easily get a lower bound on the function~$u$ in a ball centered at the origin. To do it, one can use a version of the weak Harnack inequality of Lemma~\ref{harlem} for the operator~$L$. Eventually, we get that~$u \ge \frac{1}{C_\sharp}$ in~$B_{1/2}$ for some constant~$C_\sharp \ge 1$ depending only on~$n$,~$\theta$,~$\Lambda$,~$g^\star$, and~$C_0$. In order to prove~\eqref{eq: u has linear growth, GENERAL OPERATOR}, it then suffices to build a lower barrier for~$u$ inside the annulus~$B_1 \setminus B_{1/2}$. For instance, the function
\begin{align*}
    \varphi(x):= \delta \left( e^{M (1-|x|^2)}-1  \right),
\end{align*}
for some small enough~$\delta \in (0, 1)$ and large enough~$M \ge 1$, both in dependence of~$n$,~$\theta$,~$\Lambda$,~$g^\star$, and~$C_0$ only, satisfies 
\begin{align*}
    \begin{cases}
        L \varphi \le 0 & \quad \mbox{in } B_1\setminus \overline{B}_{\frac{1}{2}},\\
        \varphi=0 & \quad \mbox{on } \partial B_1,\\
        \varphi \le \frac{1}{C_\sharp} & \quad \mbox{on } \partial B_{\frac{1}{2}}.
    \end{cases}
\end{align*}
Claim~\eqref{eq: u has linear growth, GENERAL OPERATOR} then follows from the maximum principle and estimate~\eqref{eq: u strettamente positiva in una pallina dentro B1}.

We proceed towards the proof of statements~\eqref{upealmostradsymm} and~\eqref{ualmostraddecr}. By rotating the coordinate frame, they will be established if we prove that
\begin{equation} \label{ualmostsymm2}
u(x', x_n) - u(x', - x_n) \le C \, \defi(L, g, C_0)^\alpha \quad \mbox{for all } x \in B_1 \mbox{ such that } x_n > 0
\end{equation}
and
\begin{equation} \label{ualmostdecr2}
\partial_n u(x) \le C \, \defi(L, g, C_0)^\alpha \quad \mbox{for all } x \in B_1 \mbox{ such that } x_n > 0,
\end{equation}
for some constants~$C \ge 1$ and~$\alpha \in (0, 1)$ depending only on~$n$,~$\theta$,~$\Lambda$,~$g^\star$, and~$C_0$. Note that a rotation typically leads to an equation for a different operator~$\hat{L}$ and right-hand side~$\hat{g}$. However,~$\defi(\hat{L}, \hat{g}, C_0) = \defi(L, g, C_0)$ and the coefficients of~$\hat{L}$ still satisfy assumption~\eqref{hponAandb}. For this reason, in what follows we still indicate these two objects by~$L$ and~$g$.

To prove~\eqref{ualmostsymm2}, we use the moving planes method. We do not reproduce here the full argument displayed in Steps~2-5 of the proof of Theorem~\ref{mainthm}, but only outline why the method works in this setting as well. Let~$\lambda \in (0,1)$ and consider the function~$w_\lambda (x) = u(x^\lambda) - u(x)$ for~$x \in \Sigma_\lambda$. A straightforward computation yields that~$w_\lambda$ is a solution of
$$
- \Delta w_\lambda + \tilde{c}_\lambda w_\lambda = g_\lambda + \mathscr{R}_A + \mathscr{R}_b \quad \mbox{in } \Sigma_\lambda,
$$
where
\begin{align*}
\tilde{c}_\lambda(x) & := \begin{dcases}
- \frac{g \big( {x, u(x^\lambda)} \big) - g \big( {x, u(x)} \big)}{u(x^\lambda) - u(x)} & \quad \mbox{if } u(x^\lambda) \ne u(x), \\
0 & \quad \mbox{if } u(x^\lambda) = u(x),
\end{dcases} \\
g_\lambda(x) & := g \big( {x^\lambda, u(x^\lambda)} \big) - g \big( {x, u(x^\lambda)} \big), \\
\mathscr{R}_A(x) & := \mbox{Tr} \big( {(A(x^\lambda) - I_n) D^2 u(x^\lambda)} \big) - \mbox{Tr} \big( {(A(x) - I_n) D^2 u(x)} \big), \\
\mathscr{R}_b(x) & := b(x) \cdot \nabla u(x) - b(x^\lambda) \cdot \nabla u(x^\lambda),
\end{align*}
for~$x \in \Sigma_\lambda$. A careful inspection of Steps~2-5 in the proof of Theorem~\ref{mainthm} shows that the argument goes through almost verbatim provided that the two remainder terms~$\mathscr{R}_A$ and~$\mathscr{R}_b$ can be controlled by a multiple of the deficit---one deals with~$g_\lambda$ exactly as we did with~$f_\lambda$ in~\eqref{-flambda+est}. This is indeed the case, since, recalling~\eqref{generalgradientbound},
\begin{align*}
\| \mathscr{R}_A \|_{L^\infty(\Sigma_\lambda)} + \| \mathscr{R}_b \|_{L^\infty(\Sigma_\lambda)} & \le 2 \Big( \| A - I_n \|_{L^\infty(B_1)} \| D^2 u \|_{L^\infty(B_1)} + \| b \|_{L^\infty(B_1)}  \| \nabla u \|_{L^\infty(B_1)} \Big) \\
& \le 2 C_1 \, \defi(L, g, C_0).
\end{align*}
Hence, we conclude that estimate~\eqref{ualmostsymm2} holds true.

In order to achieve~\eqref{ualmostdecr2}, we proceed as in Step~6 of the proof of Theorem~\ref{mainthm}. For the argument to work, we need~$\mathscr{R}_A$ and~$\mathscr{R}_b$ to be linearly growing away from~$\{ x_n = \lambda \}$. This follows from the estimates
\begin{align*}
|\mathscr{R}_a(x)| & \le \left| \mbox{Tr} \big( {(A(x^\lambda) - A(x)) D^2 u(x^\lambda)} \big) \right| + \left| \mbox{Tr} \big( {(A(x) - I_n) (D^2 u(x^\lambda) - D^2 u(x))} \big) \right| \\
& \le \left( [A]_{C^{0, 1}(B_1)} \| D^2 u \|_{L^\infty(B_1)} + \| A - I_n \|_{L^\infty(B_1)} \| D^3 u \|_{L^\infty(B_1)} \right) |x^\lambda - x| \\
& \le 2 C_1 \, \defi(L, g, C_0) \, (x_n - \lambda)
\end{align*}
and
\begin{align*}
|\mathscr{R}_b(x)| & \le \left| \left( b(x) - b(x^\lambda) \right) \cdot \nabla u(x) \right| + \left| b(x^\lambda) \cdot \left( \nabla u(x) - \nabla u(x^\lambda) \right) \right| \\
& \le \left( [b]_{C^{0, 1}(B_1)} \| \nabla u \|_{L^\infty(B_1)} + \| b \|_{L^\infty(B_1)} \| D^2 u \|_{L^\infty(B_1)} \right) |x^\lambda - x| \\
& \le 2 C_1 \, \defi(L, g, C_0) \, (x_n - \lambda),
\end{align*}
which hold true for all~$x \in \Sigma_\lambda$, thanks to~\eqref{generalgradientbound}. Once this is established, one obtains~\eqref{ualmostdecr2} by means of a barrier, precisely as in Step~6 of the proof of Theorem~\ref{mainthm}.
\end{proof}

As we mentioned in the introduction, Theorem~\ref{mainthm2} can be used to provide almost symmetry results for semilinear problems set in a small normal perturbation of the ball. This is the claim of Corollary~\ref{corol2}, which we establish here. 

\begin{proof}[Proof of Corollary \ref{corol2}]
Let~$v := u \circ \Psi_\epsilon$. Writing~$\Phi_\epsilon := \Psi_\epsilon^{-1}$, it is clear that~$v$ satisfies
\begin{equation*}
\begin{cases}
L[v] = f(v) & \textmd{ in } B_1 \\
v>0 & \textmd{ in } B_1 \\
v= 0 & \textmd{ in } \partial B_1 \,,
\end{cases}
\end{equation*}
where~$L$ is of the form~\eqref{defiopellittico} with
$$
A_{i j} = \sum_{k = 1}^n \left( \partial_k \Phi_\epsilon^i \circ \Psi_\epsilon \right) \left( \partial_k \Phi_\epsilon^j \circ \Psi_\epsilon \right) \quad \mbox{and} \quad b_i = - \Delta \Phi_\epsilon^i \circ \Psi_\epsilon,
$$
for~$i, j = 1, \ldots, n$. Clearly,~$\defi(L, f, C_0) \leq C \epsilon$, for some constant~$C > 0$ depending only on~$n$. Consequently, the assertion of the corollary follows by a direct application of Theorem~\ref{mainthm2}.
\end{proof}

We conclude the section with a couple of examples containing possible applications of Corollary~\ref{corol2}.

\begin{example} \label{ex1}
{\rm
Let~$\Omega_\epsilon \subset \mathbb{R}^n$ be an ellipsoid with small eccentricity. A simple computation gives that the solution of the torsional problem---i.e.,~\eqref{problemperturbedball} with~$f \equiv 1$---is explicit and its level sets coincide with dilations of~$\partial \Omega_\epsilon$. This last fact is no longer true for a general nonlinearity~$f$. However, Corollary~\ref{corol2} can be used to recover an approximate symmetry result. To see it, consider for simplicity the ellipsoid
$$
\Omega_\epsilon = \left\{ (x',x_n) \in \mathbb{R}^n: \  |x'|^2 + \frac{|x_n|^2}{a^2}  < 1 \right\},
$$
with~$a=1+\epsilon$ and~$\epsilon \in (0, 1)$. By letting~$\Psi_\epsilon : \overline{B}_1 \to \overline{\Omega}_\epsilon$ be the smooth diffeomorphism given by
$$
\Psi_\epsilon (y) = (y',ay_n) \quad \mbox{for } y \in \overline{B}_1,
$$
we clearly have that~$|\Psi_\epsilon^{-1} (x)|=r$ if and only if~$x \in r \partial \Omega_\epsilon$. In view of Corollary~\ref{corol2}, we then infer that any solution~$u \in C^2(\Omega_\epsilon) \cap C^0(\overline{\Omega}_\epsilon)$ of~\eqref{problemperturbedball} fulfilling assumption~\eqref{LinftyboundsonOmegaeps} for some~$C_0 \ge 1$ satisfies
\begin{equation} \label{ualmostradialellipsoid}
|u(p)-u(q)| \leq C\epsilon^\alpha \quad \text{for every } p,q \in r \partial \Omega_\epsilon \mbox{ and } r \in (0, 1],
\end{equation}
for some constants~$C \ge 1$ and~$\alpha \in (0, 1)$ depending only on~$n$,~$f$, and~$C_0$.
}
\end{example}

Example~\ref{ex1} can be modified to treat a general smooth perturbation~$\Omega_\epsilon$ of the unit ball, as we show here below.

\begin{example}
{\rm
Let~$\Omega_\epsilon \subset \R^n$ be a small~$C^{3, \theta}$-perturbation of the unit ball, that is
$$
\Omega_\epsilon = \Big\{ {r \big( {1 + \epsilon \varphi(x)} \big) x : x \in \partial B_1, \, r \in [0, 1)} \Big\},
$$
for some~$\varphi \in C^{3, \theta}(\partial B_1)$ with~$\| \varphi \|_{C^{3, \theta}(\partial B_1)} \le 1$ and~$\epsilon \in \left( 0, \frac{1}{2} \right]$. As in Example~\ref{ex1}, we shall show that, if~$u \in C^2(\Omega_\epsilon) \cap C^0(\overline{\Omega}_\epsilon)$ is a solution of problem~\eqref{problemperturbedball} which satisfies~\eqref{LinftyboundsonOmegaeps} for some~$C_0 \ge 1$, then~\eqref{ualmostradialellipsoid} holds true for some constants~$C \ge 1$ and~$\alpha \in (0, 1)$ depending only on~$n$,~$\theta$,~$f$, and~$C_0$, provided~$\epsilon$ is sufficiently small. To deduce this from Corollary~\ref{corol2}, we need to construct a diffeomorphism~$\widetilde{\Psi}_\epsilon$ mapping spheres centered at the origin onto dilations of~$\partial \Omega_\epsilon$. Naturally,
$$
r \partial \Omega_\epsilon = \Big\{ {r \big( {1 + \epsilon \varphi(x)} \big) x : x \in \partial B_1 } \Big\} \quad \mbox{for every } r \in (0, 1].
$$
Hence, we are led to setting~$\widetilde{\Psi}_\epsilon(x) := \big( {1 + \epsilon \varphi \big( {x / |x|} \big)} \big) x$ for~$x \in \overline{B}_1$. If~$\epsilon$ is small, this is a Lipschitz diffeomorphism of~$\overline{B}_1$ onto~$\overline{\Omega}_\epsilon$, which however fails to be more regular at the origin. In order to smooth things out, we consider a monotone non-decreasing function~$\eta \in C^\infty([0, +\infty))$ satisfying~$\eta = 0$ in~$\left[ 0, \frac{1}{4} \right]$,~$\eta = 1$ in~$\left[ \frac{1}{2}, +\infty \right)$, and define~$\Psi_\epsilon: \overline{B}_1 \to \overline{\Omega}_\epsilon$ by
$$
\Psi_\epsilon(x) := \Big( {1 + \epsilon \, \eta(|x|) \, \varphi \big( {x / |x|} \big)} \Big) x \quad \mbox{for } x \in \overline{B}_1.
$$
Now,~$\Psi_\epsilon$ is a~$C^{3, \theta}$-diffeomorphism satisfying~$\| \Psi_\epsilon - \Id \|_{C^{3, \theta}(B_1)} + \| \Psi_\epsilon^{-1} - \Id \|_{C^{3, \theta}(\Omega_\epsilon)} \le C \epsilon$, for some dimensional constant~$C$. As~$\Psi_\epsilon$ agrees with~$\widetilde{\Psi}_\epsilon$ in~$\overline{B}_1 \setminus B_{\frac{1}{2}}$, we immediately infer that~\eqref{ualmostradialellipsoid} holds true for every~$r \in \left[ \frac{1}{2}, 1 \right]$. Let then~$r \in \left( 0, \frac{1}{2} \right)$ and~$p, q \in r \partial \Omega_\epsilon$. Consider the points~$\hat{p} := \Psi_\epsilon \big( {\widetilde{\Psi}_\epsilon^{-1}(p)} \big)$ and~$\hat{q} :=\Psi_\epsilon \big( {\widetilde{\Psi}_\epsilon^{-1}(q)} \big)$. By construction,~$\big| {\Psi_\epsilon^{-1}(\hat{p})} \big| = \big| {\Psi_\epsilon^{-1}(\hat{q})} \big| = r$. Hence, by Corollary~\ref{corol2} we have that~$|u(\hat{p}) - u(\hat{q})| \le C \epsilon^\alpha$. Furthermore, by the regularity of~$u$ and the fact that both~$\Psi_\epsilon$ and~$\widetilde{\Psi}_\epsilon^{-1}$ are~$\epsilon$-close to the identity, we get that~$|u(p) - u(\hat{p})| + |u(q) - u(\hat{q})| \le C \epsilon$. Accordingly,~\eqref{ualmostradialellipsoid} is true for~$r \in \left( 0, \frac{1}{2} \right)$ as well.
}
\end{example}

\end{document}